\documentclass{amsart}
\usepackage{amsmath, amssymb,amsthm,color,pb-diagram,cite,bm}
\usepackage[initials]{amsrefs}
\usepackage[all]{xy}
\newtheorem{theorem}{Theorem}[section]
\newtheorem{proposition}[theorem]{Proposition}
\newtheorem{lemma}[theorem]{Lemma}
\newtheorem{corollary}[theorem]{Corollary}

\newtheorem{definition}{Defintion}

\newtheorem{example}{Example}
\newtheorem{problem}{Problem}
\newtheorem{question}{Question}
\usepackage{graphicx}
\theoremstyle{definition}

\newtheorem{main}{Theorem}

\newtheorem{main_cor}[main]{Corollary}

\def\Z{\mathbb{Z} }
\def\R{\mathbb{R} }
\def\Q{\mathbb{Q} }
\def\A{\mathbb{A} }
\def\T{\mathbb{T} }

\def\nbd{neighborhood }
\def\nbds{neighborhoods }
\def\Sv{\mathop{\mathrm{Sing}}(v)}
\def\Pv{\mathop{\mathrm{Per}}(v)}
\def\Cv{\mathop{\mathrm{Cl}}(v)}

\title[Generalization of Poincar\'e recurrence theorem]{Generalization of Poincar\'e recurrence theorem for flows on surfaces and characterization of minimal flows on compact surfaces}
\author{Tomoo Yokoyama}
\date{\today}
\address{Applied Mathematics and Physics Division, Gifu University, Yanagido 1-1, Gifu, 501-1193, Japan\\}
\email{tomoo@gifu-u.ac.jp}

\thanks{This work was partially supported by JSPS Kakenhi Grant Number 20K03583}


\begin{document}

\maketitle

\begin{abstract}
Poincar\'e recurrence theorem implies the density of recurrent points for volume-preserving dynamical systems on compact domains. The density of closed orbits in the non-wandering set is one of the essential properties of Axiom~A and chaos. The minimal flows are one of the most fundamental objects in topological dynamics. To analyze minimality, recurrence, and density in this paper, we introduce a strict limit circuit and a circuit with wandering holonomy. Using these concepts, we classify non-recurrent non-wandering orbits of flows on compact surfaces with finitely many singular points, characterize the minimality of flows on compact surfaces, and generalize the Poincar\'e recurrence theorem for flows on surfaces. Moreover, we characterize the density of closed orbits in the non-wandering set under finiteness of singular points. Furthermore, we demonstrate the necessity of finiteness of singular points and compactness of surfaces. In addition, the analogous results hold for flows with finitely many connected components of the singular point set on non-compact surfaces using the end compactification and collapsing singular points.
\end{abstract}

\section{Introduction}

Poincar\'e recurrence theorem grantees returns to points for volume-preserving dynamical systems on compact domains. 
In 1927, Birkhoff introduced the concepts of non-wandering points and recurrent points \cite{birkhoff1927dynamical}. 
Using these concepts, we can describe and capture dynamical behaviors. 
For instance, the existence of returns induced by Poincar\'e recurrence theorem can be described as the density of recurrent orbits in the non-wandering set for measure-preserving dynamical systems on finite measure spaces. 
The density of closed orbits in the non-wandering set for a flow is one of the essential properties for Smale's Axiom~A \cite{smale1967differentiable} and Devaney's chaos   \cite{Devaney1988chaos}. 
%
The sufficient conditions and genericity for the density of closed orbits in the non-wandering set are studied in hyperbolic dynamical systems   
\cite{Arnaud1998,dankner1977smale,dankner1978smale,kurata1978hyperbolic,mane1987proof,newhouse1973hyperbolic,pugh1967improved,pugh1968closing,smale1967differentiable}. 
It is also shown that there is a diffeomorphism with the density of closed orbits in the non-wandering set for any isotopy class of the set of diffeomorphisms on a compact manifold whose dimension is more than two. 
The analogous results for flows on manifolds of dimension at least four are also shown by suspension operations. 
Moreover, topological stability for a homeomorphism on a closed manifold and the uniform periodic shadowing property for a homeomorphism on a compact Hausdorff space are sufficient conditions for closed orbits' denseness \cite{aoki1994topological,good2018topological}. 

The characterizations for the denseness of closed orbits are known in the low-dimensional dynamical system's case. 
The following statements are equivalent for an orientation-preserving circle homeomorphism $f: \mathbb{S}^1 \to \mathbb{S}^1$: {\rm(1)} $\overline{\mathop{\mathrm{Per}}(f)} = \Omega(f)$; {\rm(2)} There are periodic orbits; {\rm(3)} The rotation number of $f$ is rational (cf. \cite[Proposition~11.1.4]{katok1997introduction}).
This equivalence means that the existence of non-periodic minimal sets is an obstruction of the density of closed orbits in the non-wandering set for a circle homeomorphism. 
Moreover, the non-existence of non-periodic recurrent points for a circle homeomorphism $f$ is equivalent to the density condition $\overline{\mathop{\mathrm{Per}}(f)} = \Omega(f)$. 
This equivalence holds for a regular curve homeomorphism, 
because the set of periodic points of a regular curve homeomorphism 
is either empty or dense in the non-wandering set \cite{daghar2021periodic}. 
Here a regular curve is a compact connected metric space such that for any point $x \in X$ and any open \nbd $U$ of $x$, there is a \nbd $V \subseteq U$ of $x$ with finite boundary. 
Furthermore, the equivalence for the existence of non-periodic minimal sets and the denseness of closed orbits is studied for local dendrites \cite{abdelli2018nonwandering,makhrova2016set,makhrova2020limit}. 


The correspondence between closed points and non-wandering points is also described. 
For instance, the following question is posed \cite[Question~1]{boyd2015diffeomorphisms}: 
\begin{question}\label{q:correspondence}
If the set of closed points is closed, under what additional conditions on the dynamical system can we conclude that any non-wandering orbits are closed?
\end{question}
The answer is negative in general. 
In fact, there is a spherical diffeomorphism with hyperbolic periodic points such that the non-wandering set is not the set of periodic points \cite{boyd2015diffeomorphisms}. 
On the other hand, there are affirmative results for mappings on closed intervals and $C^1$ self-maps $f$ of form $f(x,y) = (f_1(x), f_2(x,y))$ on the unit square as follows. 
A continuous map of a closed interval to itself whose periodic point set is closed and satisfies the denseness of closed orbits, because the following are equivalent for a continuous map $f$ of a closed interval to itself: {\rm(1)} $\mathop{\mathrm{Per}}(f)$ is closed; {\rm(2)} $\mathop{\mathrm{Per}}(f) = \Omega(f)$, where $\mathop{\mathrm{Per}}(f)$ is the set of periodic points of $f$ and $\Omega(f)$ is the set of non-wandering points \cite{xiong1981continuous,nitecki1982maps}. 
The same statement holds for a $C^1$ self-map $f(x,y) = (f_1(x), f_2(x,y))$ on the unit square such that 
$\mathop{\mathrm{Per}}(f)$ is closed \cite{arteaga1995smooth,efremova2014remarks}.

Minimality is studied from various aspects in topological dynamics (see books  \cite{Auslander1988mini,Ellis1969lecture,Gottschalk1955top_dyn} for details). 
For instance, examples of minimal flows on surfaces 
are constructed \cite{aranson1973invariant,gardiner1985structure,Sacker1972existence}, and properties of minimal flows on 
surfaces are described \cite{Athanassopoulos1995minimal,Athanassopoulos1997stable,Basener2006min,Gottschalk1963minimal,Smith1988transitive,Ulcigrai2011absence}. 
On the other hand, the existence of dense orbits of flows on compact surfaces is topologically characterized \cite{Marzougui2009dense,yokoyama2016topological}. 
Similarly, we would like to find a topological characterization of minimality for flows on compact surfaces. 
Therefore we pose the following problem. 
\begin{problem}\label{prob:minimality}
Find a topological characterization of minimal flows. 
\end{problem}

From physical and differential equation's points of view, the no-slip boundary condition and non-compactness appearing from properties of fluid phenomena on punctured spheres are fundamental conditions. 
In order to analyze such fluid phenomena, it is necessary to allow the degeneracy of singularities and non-compactness of surfaces. 

To analyze minimality, recurrence, and density in this paper, we introduce a strict limit circuit and a circuit with wandering holonomy. Using these concepts, we characterize the mechanism of occurrence of the non-recurrent non-wandering behaviors as follows. 

\begin{main}\label{main:rec_wandering_pt}
Let $v$ be a flow with finitely many connected components of the singular point set on a compact surface. 
Then one of the following three statements holds for an orbit $O$ exclusively: 
\\
{\rm(1)} The orbit $O$ is recurrent {\rm (i.e.} $O \subseteq \Cv \sqcup \mathrm{R}(v)$ {\rm)}. 
\\
{\rm(2)} The orbit $O$  is not recurrent but non-wandering and satisfies $\overline{O} - O = \alpha(O) \cup \omega(O) \subseteq \Sv$ and either {\rm(2.1)}, {\rm(2.2)}, {\rm(2.3)}, or {\rm(2.4)}: 
\\
{\rm(2.1)} There is a strict limit quasi-circuits in $\Omega(v)$ containing $O \subset \mathrm{int} \mathrm{P}(v)$. 
\\
{\rm(2.2)} There is the blow-up of a circuit with wandering holonomy containing $O \subset \mathrm{int} \mathrm{P}(v)$. 
\\
{\rm(2.3)} There is the blow-up of a circuit which is contained in $\overline{\Pv} - \Pv \subseteq \Omega(v)$ and contains $O$. 
\\
{\rm(2.4)} $O \subset \overline{\mathrm{R}(v)} - \mathrm{R}(v)$. 
\\
{\rm(3)} The orbit $O$ is wandering {\rm(i.e.} $O \cap \Omega(v) = \emptyset${\rm)}. 
\end{main}

The previous theorem follows from a more general result for (possibly non-compact) surfaces with finite genus and finite ends (see Theorem~\ref{main:rec_wandering}). 
%
%
Applying Theorem~\ref{main:rec_wandering_pt}, we also obtain the following characterization of minimality, which is an answer to Problem~\ref{prob:minimality} for flows on surfaces. 

\begin{main}\label{main:minimal}
A flow on a connected surface with finite genus and finite ends is minimal if and only if it is non-wandering and the set difference $\overline{O} - O$ for any orbit $O$ does not consist of singular points. 
\end{main}

Each of the non-wandering condition and the condition for the difference is necessary because of the existence of the Denjoy flow and periodic flows. 
Indeed, the Denjoy flow is not minimal, but the set difference $\overline{O} - O$ for any orbit $O$ does not consist of singular points, and any periodic flows are non-wandering but the set difference $\overline{O} - O = \emptyset$ for any orbit $O$ consists of singular points.

Applying Theorem~\ref{main:rec_wandering_pt}, we have the following characterization of the density of recurrence in the non-wandering set, which is a generalization of the Poincar\'e recurrence theorem for flows on surfaces.

\begin{main}\label{main_cor:rec_th}
Then following are equivalent for a flow with finitely many connected components of the singular point set on a compact surface: 
\\
{\rm(1)} The set of recurrent points is dense in the non-wandering set.
\\
{\rm(2)} There are neither non-closed recurrent points, strict limit quasi-circuits, nor blow-ups of circuits with wandering holonomy. 
\end{main}

Theorem~\ref{main_cor:rec_th} follows from a more general result for surfaces with finite genus and finite ends (see Theorem~\ref{thm:rec_th}). 
Theorem~\ref{main_cor:rec_th} can be reduced into the following statement. 

%
%


\begin{main_cor}\label{main:a}
The following conditions are equivalent for a flow $v$ with finitely many connected components of the set of singular points on a compact surface:
\\
$(1)$ $\overline{\mathop{\mathrm{Cl}}(v)} = \Omega(v)$.
\\
$(2)$ There are neither non-closed recurrent points, strict limit circuits, nor circuits with wandering holonomy.
\end{main_cor}
 
The compactness of the surface is necessary (see Example~\ref{ex:counter_exam01} for details). 
The finite existence of connected components of the singular point set in Theorem~\ref{main_cor:rec_th} and Corollary~\ref{main:a} is necessary (see Lemma~\ref{lem:counter_exam01} for details). 
Notice that the non-existence of three conditions is necessary. 
In fact, the non-existence of non-closed recurrent points is necessary because of the existence of irrational rotations on tori. 
The non-existence of strict limit circuits is also necessary even for $\Omega$-stable spherical flows (see Example~\ref{ex:counter_exam_strict}). 
The necessity of the non-existence of circuits with wandering holonomy follows from an example in the proof of Lemma~\ref{lem:counter_exam02}. 
%
In addition, we show the analogous results for flows on surfaces with finite genus (see Corollary~\ref{thm:main} for details).

Applying Corollary~\ref{main:a}, we obtain the following affirmative statement for \cite[Question~1]{boyd2015diffeomorphisms} (i.e. Question~\ref{q:correspondence}), which generalizes the above result for a continuous map of a closed interval \cite{xiong1981continuous,nitecki1982maps}. 

\begin{main_cor}\label{main_cor:a}
The following conditions are equivalent for a flow $v$ with finitely many connected components of the set of singular points on a compact surface $S$:
\\
$(1)$ $\mathop{\mathrm{Cl}}(v) = \Omega(v)$.
\\
$(2)$ The closed point set $\mathop{\mathrm{Cl}}(v)$ is closed, and there are neither non-closed recurrent points, non-periodic limit circuits, nor circuits with wandering holonomy.
\\
$(3)$ The set $\mathrm{P}(v)$ of non-recurrent points are open, and there are neither non-closed recurrent points, non-periodic limit circuits, nor circuits with wandering holonomy.
In any case, we have $S = \mathop{\mathrm{Cl}}(v) \sqcup \mathrm{P}(v) = \Omega(v) \sqcup \mathrm{P}(v)$.
\end{main_cor}

The finite existence of connected components of the singular point set in Corollary~\ref{main_cor:a} is necessary (see Lemma~\ref{lem:counter_exam01} for details). 
%
By the end completion of surfaces of finite genus, we can obtain an analogous result for surfaces of finite genus (see Corollary~\ref{thm:main02} for details).

Furthermore, Theorem~\ref{main:rec_wandering_pt} implies the following characterization of the finiteness of the non-wandering set. 

\begin{main_cor}\label{main:c}
The following conditions are equivalent for a flow $v$ on a compact surface:
\\
{\rm(1)} The non-wandering set $\Omega(v)$ consists of finitely many orbits.
\\
{\rm(2)} Each recurrent orbit is a closed orbit, any limit circuits consist of at most finitely many orbits,  and there are at most finitely many closed orbits and limit circuits. 
\\
{\rm(3)} Each recurrent orbit is a closed orbit, any limit circuits consist of at most finitely many orbits,  and there are at most finitely many closed orbits and strict limit circuits. 
\\
In any case, the non-wandering set $\Omega(v)$ consists of closed orbits, strict limit circuits, and circuits with wandering holonomy. 
\end{main_cor}

The present paper consists of eight sections.
In the next section, as preliminaries, we introduce fundamental concepts.
In \S 3, we show a classification of a non-recurrent point under the non-existence of non-closed recurrent orbits (Lemma~\ref{lem:wandering_pt}), which is a key lemma of this paper. 
In \S 4, we topologically characterize orbits for flows with finitely many singular points on compact surfaces.
In \S 5, we generalized the results in the previous section into flows with finitely many connected components of the singular point set on surfaces with finite genus and finite ends.
In \S 6, a characterization of the finiteness of the non-wandering set is described, and correspondence between the closed point set and the non-wandering set is characterized. 
In \S 7, some examples of flows are described to show the necessity of conditions in Theorem~\ref{main_cor:rec_th} and Corollary~\ref{main:a} and lemmas. 
%
%
In the final section, we observe the density for non-wandering flows on compact surfaces. 

\section{Preliminaries}

\subsection{Topological notion}
Denote by $\overline{A}$ the closure of a subset $A$ of a topological space, by $\mathrm{int}A$ the interior of $A$, and by $\partial A := \overline{A} - \mathrm{int}A$ the boundary of $A$, where $B - C$ is used instead of the set difference $B \setminus C$ when $B \subseteq C$.
A boundary component of a subset $A$ is a connected component of the boundary of $A$. 
A subset is locally dense if its closure has a nonempty interior. 
Recall that a boundary component of a subset $A$ is a connected component of the boundary of $A$. 
%
%
\subsubsection{Curves and loops}
A curve is a continuous mapping $C: I \to X$ where $I$ is a non-degenerate connected subset of a circle $\mathbb{S}^1$. 
Here a non-degenerate subset is a subset containing at least two points. 
A curve is simple if it is injective.
We also denote by $C$ the image of a curve $C$.
Denote by $\partial C := C(\partial I)$ the boundary of a curve $C$, where $\partial I$ is the boundary of $I \subset \mathbb{S}^1$. Put $\mathrm{int} C := C \setminus \partial C$. 
A simple curve is a simple closed curve if its domain is $\mathbb{S}^1$ (i.e. $I = \mathbb{S}^1$).
A simple closed curve is also called a loop. 
Two disjoint loops are parallel if there is an open annulus whose boundary is the union of the two loops. 
An arc is a simple curve whose domain is an interval. 
An orbit arc is an arc contained in an orbit. 

\subsubsection{Surfaces}
By a surface,  we mean a two-dimensional paracompact manifold, that does not need to be orientable.

\subsection{Notions of dynamical systems}
A flow is a continuous $\R$-action on a manifold.
From now on, we suppose that flows are on surfaces unless otherwise stated.
Let $v : \R \times S \to S$ be a flow on a surface $S$.
For $t \in \R$, define $v_t : S \to S$ by $v_t := v(t, \cdot )$.
For a point $x$ of $S$, we denote by $O(x)$ the orbit of $x$, $O^+(x)$ the positive orbit (i.e. $O^+(x) := \{ v_t(x) \mid t > 0 \}$), and $O^-(x)$ the negative orbit (i.e. $O^-(x) := \{ v_t(x) \mid t < 0 \}$).
Recall that a point $x$ of $S$ is singular if $x = v_t(x)$ for any $t \in \R$, is non-singular if $x$ is not singular, and is periodic if there is a positive number $T > 0$ such that $x = v_T(x)$ and $x \neq v_t(x)$ for any $t \in (0, T)$.
An orbit is singular (resp. periodic) if it contains a singular (resp. periodic) point. 
An orbit is closed if it is singular or periodic.
Denote by $\mathop{\mathrm{Sing}}(v)$ the set of singular points and by $\mathop{\mathrm{Per}}(v)$  the union of periodic orbits.
The union $\bm{\mathop{\mathrm{Cl}}(v)} := \mathop{\mathrm{Sing}}(v) \sqcup \mathop{\mathrm{Per}}(v)$ is the union of closed orbits, where $\sqcup$ denotes a disjoint union. 
Recall that the $\omega$-limit (resp. $\alpha$-limit)  set of a point $x$ is $\omega(x) := \bigcap_{n\in \mathbb{R}}\overline{\{v_t(x) \mid t > n\}}$ (resp.  $\alpha(x) := \bigcap_{n\in \mathbb{R}}\overline{\{v_t(x) \mid t < n\}}$).
For an orbit $O$, define $\omega(O) := \omega(x)$ and $\alpha(O) := \alpha(x)$ for some point $x \in O$.
Note that an $\omega$-limit (resp. $\alpha$-limit) set of an orbit is independent of the choice of point in the orbit. 
A point $x$ of $S$ is {\bf recurrent} if $x \in \omega(x) \cup \alpha(x)$.
Denote by $\mathrm{R}(v)$ the set of non-closed recurrent points.
A point is non-wandering if for any its neighborhood $U$ and for any positive number $N$, there is a number $t \in \mathbb{R}$ with $|t| > N$ such that $v_t(U) \cap U \neq \emptyset$.
Denote by $\bm{\Omega (v)}$ the set of non-wandering points, called the non-wandering set. 
A flow is non-wandering if every point is non-wandering.  

A subset $A$ is positive (resp. negative) invariant if $v_t(A) \subseteq A$ for any $t \in \R_{>0}$ (resp. $t \in \R_{<0}$). 
A subset is invariant (or saturated) if it is a union of orbits. 
The saturation of a subset is the union of orbits intersecting it.
Denote by $v(A)$ the saturation of a subset $A$.
Then $v(A) = \bigcup_{a \in A} O(a)$.
A nonempty closed invariant subset is minimal if it contains no proper nonempty closed invariant subsets. 

A non-singular orbit is a separatrix if there is a singular point which is either its $\alpha$-limit set or its $\omega$-limit set.

\subsubsection{Topological properties of orbits}

An orbit is proper if it is embedded and exceptional if it is neither proper nor locally dense. 
A point is proper (resp. locally dense, exceptional) if its orbit is proper (resp. locally dense, exceptional).
Denote by $\mathrm{LD}(v)$ (resp. $\mathrm{E}(v)$, $\mathrm{P}(v)$) the union of locally dense orbits (resp. exceptional orbits, non-closed proper orbits). 
Note that an orbit on a paracompact manifold (e.g. a surface) is proper if and only if it has a neighborhood in which the orbit is closed \cite{yokoyama2019properness}.  
This implies that a non-recurrent point is proper and so that a non-proper point is recurrent.  
In \cite[Theorem~VI]{cherry1937topological}, Cherry showed that the closure of a non-closed recurrent orbit $O$ of a flow on a manifold contains uncountably many non-closed recurrent orbits whose closures are $\overline{O}$. 
This means that each non-closed recurrent orbit of a flow on a manifold has no neighborhood in which the orbit is closed, and so is not proper. 
In particular, a non-closed proper orbit is non-recurrent. 
Therefore the union $\mathrm{P}(v)$ of non-closed proper orbits is the set of non-recurrent points and that $\mathrm{R}(v) = \mathrm{LD}(v) \sqcup \mathrm{E}(v)$.
Hence we have a decomposition $S = \mathop{\mathrm{Sing}}(v) \sqcup \mathop{\mathrm{Per}}(v) \sqcup \mathrm{P}(v) \sqcup \mathrm{R}(v) = \mathop{\mathrm{Cl}}(v) \sqcup \mathrm{P}(v) \sqcup \mathrm{R}(v)$. 

Recall that a non-wandering flow on a compact surface has no exceptional orbits (i.e. $\mathrm{E}(v) = \emptyset$) because of Lemma~\ref{lem:top23} stated below.


\subsubsection{Flox box}
Let $I, J$ be intervals which are either $(0,1)$, $(0,1]$, $[0,1)$, or $[0,1]$ and $\mathbb{D} := I \times J \subset \R^2$ be a disk whose orbits are of form $I \times \{ t \}$ for some $t \in J$. 
Then $D$ is called a trivial flow box. 
A disk $B$ is a flow box if there is an annular \nbd $U_B$ of the boundary $\partial B$ such that the intersection $B \cap U_B$ is the intersection $D \cap U_D$ of a trivial flow box $D$ and its neighborhood $U_D$ of the boundary $\partial D$ up to topological equivalence as in Figure~\ref{fig:flowbox}. 
In other words, a disk $B$ is a flow box if there is a homeomorphism $f \colon B \cap U_B \to D \cap U_D$ such that the images of orbit arcs are orbit arcs. 
\begin{figure}
\begin{center}
\includegraphics[scale=0.35]{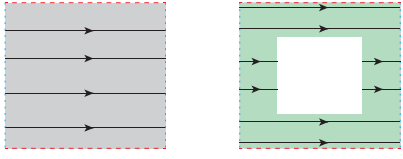}
\end{center}
\caption{Left, a trivial flow box; right, the intersection of a flow box and its small neighborhood of the boundary.}
\label{fig:flowbox}
\end{figure}

\subsubsection{Circuits}
A trivial circuit is a singular point. 
An open annular subset $\mathbb{A}$ of a surface is a collar of a singular point $x$ if the union $\mathbb{A} \sqcup \{ x \}$ is a neighborhood of $x$. 
A trivial circuit $x$ is attracting (resp. repelling) if there is its collar which is contained in the stable (resp. unstable) manifold of $x$.  
In other words, an attracting trivial circuit is either a $\partial$-source or a source, and a repelling trivial circuit is either a $\partial$-sink or a sink. 
Here a $\partial$-source (resp. $\partial$-sink) is a singular point whose lift to the double of the surface is a source (resp. sink) with respect to the lift of the flow.  
By a cycle or a periodic circuit, we mean a periodic orbit.

\begin{definition}
An image of an oriented circle by a continuous orientation-preserving mapping is a {\bf non-trivial circuit} if it  is either a cycle or a directed graph but not a singleton and which is the union of separatrices and finitely many singular points.
\end{definition}

A {\bf circuit} is either a trivial or non-trivial periodic circuit.
An open annular subset $\mathbb{A}$ of a surface is a collar of a non-trivial circuit $\gamma$ if $\gamma$ is a boundary component of $\mathbb{A}$ and there is a neighborhood $U$ of $\gamma$ such that $\mathbb{A}$ is a connected component of the complement $U - \gamma$. 
A circuit $\gamma$ is a semi-attracting (resp. semi-repelling) circuit with respect to a small collar $\A$ if $\omega(x)  = \gamma$ (resp. $\alpha(x) = \gamma$) and $O^+(x) \subset \A$ (resp. $O^-(x) \subset \A$) for any point $x \in \A$. 
Then $\A$ is called an attracting (resp. a repelling) collar basin of $\gamma$.

\begin{definition}
A non-trivial circuit $\gamma$ is a {\bf limit circuit} if it is a semi-attracting or semi-repelling circuit. 
\end{definition}

Notice that any limit circuits consist of closed orbits and non-recurrent orbits (see Figure~\ref{fig:limit_circuit}. 
\begin{figure}
\begin{center}
\includegraphics[scale=0.28]{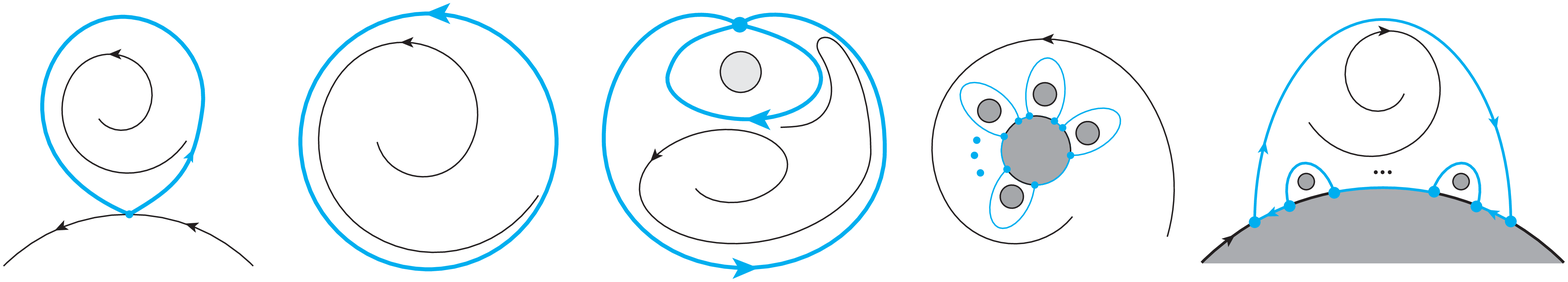}
\end{center}
\caption{Examples of limit circuits.}
\label{fig:limit_circuit}
\end{figure} 
A limit cycle is a limit circuit in $\mathop{\mathrm{Per}}(v)$.
A non-trivial circuit $\gamma$ is {\bf one-sided} if for any small neighborhood $U$ of $\gamma$ there is a collar $V \subset U$ of $\gamma$ such that the union $V \sqcup \gamma$ is a neighborhood of some point in $\mathrm{P}(v) \cap \gamma$. 
Notice that a non-trivial circuit $\gamma$ is not one-sided if and only if there is a small neighborhood $U$ of $\gamma$ such that the union $V \sqcup \gamma$ for any collar $V \subset U$ of $\gamma$ is not a neighborhood of any point in $\mathrm{P}(v) \cap \gamma$. 
For a circuit $\mu$ which is a simple closed curve, notice that the circuit $\mu$ is one-sided if and only if it is either a boundary component of a surface or has a small neighborhood which is a M\"obius band, and that the circuit $\mu$ is not one-sided if and only if it has an open small annular neighborhood $\mathbb{A}$ such that the complement $\mathbb{A} - \mu$ consists of two open annuli.

\subsection{Strict limit circuit and circuit with wandering holonomy}
To characterize the density of closed orbits in the non-wandering set, we introduce strict limit circuits and circuits with wandering holonomy as follows. 

\begin{definition}
A non-trivial non-periodic limit circuit $\gamma$ is a {\bf strict limit circuit} if either $\gamma$ is one-sided or there are a separatrix $\mu \subseteq \gamma$ and a transverse open arc $T$ intersecting $\mu$ such that $f_v|_{T_1}$ is either attracting or repelling, and that $f_v|_{T_2}$ has no fixed points,  where $f_v :T \to T$ is the first return map, $T_1$ and $T_2$ are the two connected components of $T \setminus \mu$.
\end{definition}
\begin{figure}
\begin{center}
\includegraphics[scale=0.35]{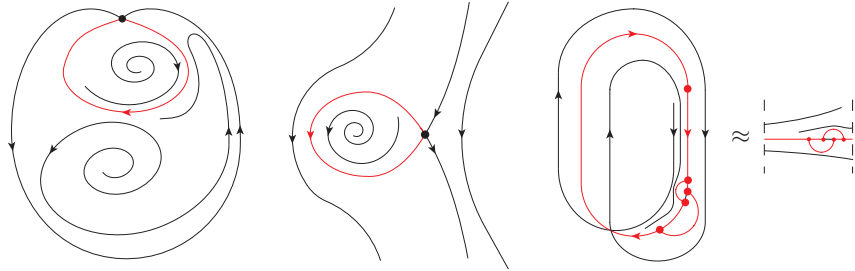}
\end{center}
\caption{Examples of strict limit circuits.}
\label{fig:strict_limit_circuits}
\end{figure}
Notice that any one-sided strict limit circuit is either attracting or repelling, as in Figure~\ref{fig:strict_limit_circuits}. 
Note that there is a strict limit circuit with infinitely many edges and such that any non-trivial circuit in it contains non-closed proper orbits as in Figure~\ref{NAC02}.
\begin{figure}
\begin{center}
\includegraphics[scale=0.2]{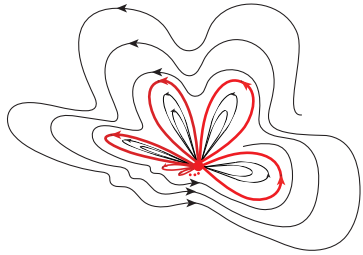}
\end{center}
\caption{A strict limit circuit that consists of a degenerate singular point and infinitely many connecting separatrices, and its \nbd which consists of a singular point and non-recurrent orbits.}
\label{NAC02}
\end{figure} 


\begin{definition}
A non-trivial circuit $\gamma$ is a {\bf circuit with wandering holonomy} if there are a non-singular point $x \in \gamma$ and arbitrarily small open transverse arc $I$ containing $x$ such that the first return map on $I$ is orientation reversing, the domain is nonempty, and the intersection of the domain and the codomain is empty. 
\end{definition}

For instance, there are flows with circuits with wandering holonomy as in Figure~\ref{nondir} and Figure~\ref{g2-ex}. 
\begin{figure}
\begin{center}
\includegraphics[scale=0.2]{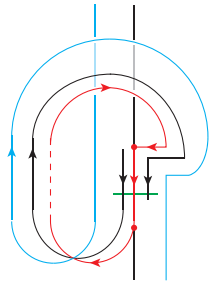}
\end{center}
\caption{An example of a circuit with wandering holonomy.}
\label{nondir}
\end{figure} 
We have the following observation. 

\begin{lemma}\label{lem:wandering_hol}
Let $v$ be a flow on a surface $S$. 
Then each circuit with wandering holonomy is not periodic and contains points in $\Omega(v) - \overline{\Cv}$. 
\end{lemma}

\begin{proof}
Let $\gamma$ be a circuit with wandering holonomy.  
We claim that $\gamma$ is non-periodic. 
Indeed, assume that $\gamma$ is periodic. 
Fix a small open transverse arc $I$ intersecting $\gamma$. 
By the flow box theorem (cf. \cite[Theorem 1.1, p.45]{aranson1996introduction}) to $\gamma$, the intersection of the domain and the codomain for the first return map to $I$ contains a nonempty open interval, which contradicts that the intersection of the domain and the codomain is empty. 

Fix a non-singular point $x \in \gamma$ as in the definition of circuit with wandering holonomy. 
The flow box theorem to $x$ implies that there is a \nbd of $x$ containing no singular points. 
Since the intersection of the domain and the codomain for the first return map to the arbitrarily small open transverse arc $I$ containing $x$ contains a nonempty open interval, the point $x$ is non-wandering. 
The empty intersection of the domain and the codomain of the first return map implies that there is a \nbd of $x$ containing no periodic points. 
This means that $x \in \Omega(v) - \overline{\Cv}$. 
\end{proof}

\section{Topological characterization of density of periodic orbits in the non-wandering set}

This section topologically characterizes the denseness of closed orbits for a flow with finitely many singular points on compact surfaces. 
We recall the following statement, which is used to show several results. 

\begin{lemma}[Lemma~2.3 \cite{yokoyama2016topological}]\label{lem:top23}
 Let $v$ be a flow on a compact surface $S$. 
 Then $\overline{\Cv \sqcup \mathrm{LD}(v)} \cap \mathrm{E}(v) = \emptyset$ and $\overline{\Cv \sqcup \mathrm{E}(v)} \cap \mathrm{LD}(v) = \emptyset$. 
Moreover, we have $\mathrm{E}(v) \subseteq \mathrm{int} \overline{\mathrm{P}(v)}$.  
\end{lemma}

Note that the compactness in the previous lemma is necessary (see Example ~\ref{ex:counter_exam01} for details).

\subsection{Necessary conditions}
We have the following necessary conditions for the denseness of closed orbits with respect to a flow on a surface. 

\begin{proposition}\label{prop:necessary_dense}
Let $v$ be a flow on a compact surface. 
If $\overline{\mathop{\mathrm{Cl}}(v)} = \Omega(v)$, then there are neither non-closed recurrent points, strict limit circuits, nor circuits with wandering holonomy.
\end{proposition}

This proposition follows from Lemma~\ref{thm41} and Lemma~\ref{thm41a}. 

\begin{lemma}\label{thm41}
Let $v$ be a flow on a compact surface $S$.
If $\overline{\mathop{\mathrm{Cl}}(v)} = \Omega(v)$, then $\mathrm{R}(v) = \emptyset$ $( \mathrm{i.e.}$ $S = \mathop{\mathrm{Cl}}(v) \sqcup \mathrm{P}(v)$ $)$. 
\end{lemma}

\begin{proof}
Suppose that $\overline{\mathop{\mathrm{Cl}}(v)} = \Omega(v)$. 
Since recurrent orbits are non-wandering, we have $\mathrm{R}(v) \subseteq  \Omega(v) = \overline{\mathop{\mathrm{Cl}}(v)}$.
From the compactness of $S$, Lemma~\ref{lem:top23} implies that $\overline{\mathop{\mathrm{Per}}(v)} \cap \mathrm{R}(v) = \emptyset$.
The closedness of the singular point set $\mathop{\mathrm{Sing}}(v)$ implies that $\overline{\mathop{\mathrm{Cl}}(v)} \cap \mathrm{R}(v) = \emptyset$ and so that $\mathrm{R}(v) = \emptyset$.
Then $S = \mathop{\mathrm{Cl}}(v) \sqcup \mathrm{P}(v)$. 
\end{proof}

The converse of Lemma~\ref{thm41} is not true in general (see Lemma~\ref{lem:counter_exam00} for details). 
We show that the denseness of closed orbits implies the non-existence of circuits with wandering holonomy and strict limit circuits. 

\begin{lemma}\label{thm41a}
Let $v$ be a flow on a compact surface $S$.
If $\overline{\mathop{\mathrm{Cl}}(v)} = \Omega(v)$, then there are neither strict limit circuits nor  circuits with wandering holonomy.
\end{lemma}

\begin{proof}
From $\overline{\mathop{\mathrm{Cl}}(v)} = \Omega(v)$, Lemma~\ref{lem:wandering_hol} imlies that there are no circuits with wandering holonomy.
By Lemma~\ref{thm41}, there are no non-closed recurrent orbits and so $S = \mathop{\mathrm{Cl}}(v) \sqcup \mathrm{P}(v)$. 
Then $\Omega(v) \cap \mathrm{int} \mathrm{P}(v) = \Omega(v) - \overline{\mathop{\mathrm{Cl}}(v)}$.
We claim that there are no strict limit circuits. 
Indeed, assume that there is a strict limit circuit $\gamma$.
By definition, from $S = \mathop{\mathrm{Cl}}(v) \sqcup \mathrm{P}(v)$, the circuit $\gamma$ contains a point in $\Omega(v) \cap \mathrm{int} \mathrm{P}(v) = \Omega(v) - \overline{\mathop{\mathrm{Cl}}(v)}$ and so $\overline{\mathop{\mathrm{Cl}}(v)} \neq \Omega(v)$, which contradicts $\overline{\mathop{\mathrm{Cl}}(v)} = \Omega(v)$. 
\end{proof}

The converse of Lemma~\ref{thm41a} is not true in general (see Lemma~\ref{lem:counter_exam01}
 for details). 
We have the following statements. 

\begin{corollary}\label{prop:necessary_dense02}
Let $v$ be a flow on a compact surface. 
If $\mathop{\mathrm{Cl}}(v) = \Omega(v)$, then $\mathop{\mathrm{Cl}}(v)$ is closed, and there are neither non-closed recurrent points, non-periodic limit circuits, nor circuits with wandering holonomy.
\end{corollary}

\begin{proof}
Suppose that $\mathop{\mathrm{Cl}}(v) = \Omega(v)$. 
Since $\Omega(v)$ is closed, the closed point set $\mathop{\mathrm{Cl}}(v)$ is closed with $\overline{\mathop{\mathrm{Cl}}(v)} = \mathop{\mathrm{Cl}}(v) = \Omega(v)$. 
Proposition~\ref{prop:necessary_dense} implies that there are neither non-closed recurrent points, strict limit circuits, nor circuits with wandering holonomy.
Since any points in $\omega$-limit and $\alpha$-limit sets of any points are non-wandering, there are no non-periodic limit circuits. 
\end{proof}

\subsection{Sufficient conditions under the finite existence of singular points}

In this subsection, we show that any non-singular points in the boundary of the periodic point set are contained in a circuit under the finite existence of singular points as follows. 

\begin{lemma}\label{lem:circuit}
Let $v$ be a flow with finitely many singular points on a compact surface $S$. 
Each non-singular point in $\overline{\Pv}$ is contained in a circuit. 
\end{lemma}

To show the previous lemma, we show the following existence of a closed transversal parallel to a limit circuit. 

\begin{lemma}\label{lem:lc_transersal}
For any limit circuit, there is a closed transversal that is not contractible in its associated collar. 
\end{lemma}

\begin{proof}
Let $\gamma_0$ be a limit circuit for a flow $v$ on a surface. 
By time reversion if necessary, we may assume that $\gamma_0$ is semi-attracting. 
Let $x_0 \notin \gamma_0$ be a point with $\gamma_0 = \omega(x_0)$, $\A$ its small associated collar, $y \in \gamma_0$ a non-recurrent point, $I \subset \A$ an oriented open transverse arc such that $x_0$ and $y$ are boundary components of $I$, and $f_v: I \to I$ the first return map on $I$ induced by $v$, $x_i := (f_v)^i(x_0)$ the $i$-th return of $x_0$, $C_{a,b} \subset O^+(x)$ the orbit arc from $a$ to $b$, and $I_{a,b} \subset I$ the subinterval from $a$ to $b$ of $I$.  
Then $I_{a,b}$ and $I_{b,a}$ are the same interval but have the opposite directions.
We may assume that $x_0 < x_1$. 

We claim that if the restriction of $f_v$ to a neighborhood of $x_i$ for some $i \in \Z_{\geq 0}$ is orientation-preserving then there is a desired closed transversal. 
Indeed, the restriction of $f_v$ to a neighborhood of $x_i$ for some $i \in \Z_{\geq 0}$ is orientation-preserving. 
Then put $C := C_{x_i, x_{i+1}}$ and $J := I_{x_i, x_{i+1}}$. 
By the waterfall construction (cf. \cite[Lemma~3.3.7 p.86]{candel2000foliation}) to the loop $\mu := C \cup J$ (see Figure~\ref{wf}), there is a closed transversal $\gamma$ intersecting $O(x)$ near $\mu$ which is parallel to $\partial \A$. 
\begin{figure}
\begin{center}
\includegraphics[scale=0.4]{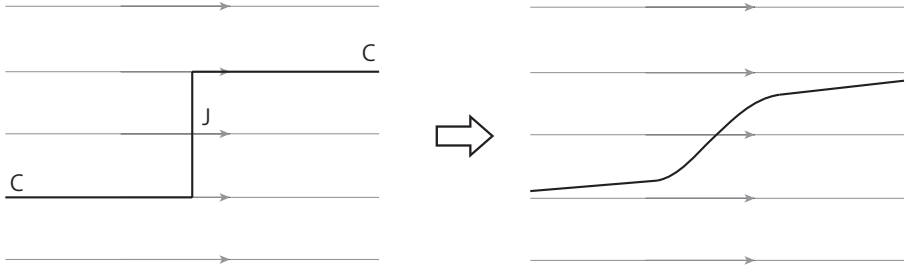}
\end{center}
\caption{The waterfall construction}
\label{wf}
\end{figure}

Thus we may assume that the restriction of $f_v$ to a neighborhood of $x_i$ for any $i \in \Z_{\geq 0}$ is orientation-reversing.  
We claim that there is a natural number $i$ such that $x_{i+1} < x_i$. 
Indeed, otherwise $x_i < x_{i+1}$ for any $i \in \Z_{\geq 0}$. 
Then each pair of loops $\gamma_i := C_{x_{2i}, x_{2i+1}} \cup I_{x_{2i+1}, x_{2i}}$ has disjoint neighborhoods each of which is a M\"obius band. 
This means that $S$ has infinitely many non-orientable genus, which contradicts the compactness of $S$.

By renumbering, we may assume that $x_2 < x_1$.  
From $x_0 < x_1$, the interval $I_{x_2, x_0}$ does not intersect $x_1$. 
Then the first return map for $I_{x_2, x_0}$ along $C_{x_0, x_2}$ is orientation-preserving such that a pair of $C := C_{x_0, x_2}$ and $J := I_{x_2, x_0}$. 
As above, the waterfall construction to the loop $\mu := C \cup J$  there is a closed transversal $\gamma$ intersecting $O(x)$ near $\mu$ which is parallel to $\partial \A$. 
\end{proof}

%

We demonstrate Lemma~\ref{lem:circuit} as follows.

\begin{proof}[Proof of Lemma~\ref{lem:circuit}]
By Gutierrez's smoothing theorem~\cite{gutierrez1978structural}, we may assume that $v$ is $C^1$.  
Fix a non-singular point $x_0 \in \overline{\Pv}$. 
Then there are a small open transverse arc $I$ whose boundary contains $x_0$, and a sequence $(x_i)_{i \in \Z_{>0}}$ of fixed points of the first return map on $I$ which converges to $x_0$ such that one of the boundary components of the saturation $v(I)$ is the periodic orbit $O(x_1)$. 
Let $I_{a,b}$ be the open subarc in $I$ from a point $a \in I$ to a point $b \in I$. 
Denote by $\A_i$ the connected component of $S - \bigcup_{j \in \Z_{>0}} O(x_j)$ containing $I_{x_i,x_{i+1}}$.  
By the finite existence of genus and singular points, taking a subsequence of $(x_i)_{i \in \Z_{> 0}}$, we may assume that any $\A_i$ is an open invariant annulus intersecting no singular points. 
Then the union $\A := \bigcup_{i \in \Z_{>0} } \A_i \sqcup O(x_{i+1})$ is an open invariant annulus intersecting no singular points. 
Let $\partial$ be the boundary component of $\A$ containing $x_0$. 
Then $\partial \A = O(x_1) \sqcup \partial$. 
Lemma~\ref{lem:top23} implies that $\overline{\mathop{\mathrm{Per}}(v)} \cap \mathrm{R}(v) = \emptyset$ and so that $\mathrm{R}(v) \cap \partial = \emptyset$.
Since $v$ is $C^1$, there is a continuous vector field $X$ generating $v$. 
Using a bump function on a flow box $U$ with $I \subset U \subset \A$ and $x_0 \in \partial U$, we can modify the vector field $X$ into the resulting vector field generating a continuous flow $w$ such that the first return map on $I$ by $w$ is a semi-attracting to $x_0$. 
By construction, we obtain $v|_{S - \A} = w|_{S - \A}$. 
Moreover, we may assume that $\A \cap \mathop{\mathrm{Sing}}(w) = \emptyset$. 
From $\mathrm{R}(v) \cap \partial = \emptyset$ and $\partial \subset S - \A$, we have $\mathrm{R}(w) \cap \partial = \emptyset$. 
Therefore $\partial$ is the $\omega$-limit set of the flow $w$ with finitely many singular points. 
Since a non-singular point $x_0 \notin \mathop{\mathrm{Sing}}(w)$ is contained in $\partial$,  by $\mathrm{R}(w) \cap \partial = \emptyset$, the generalization of the Poincar\'e-Bendixson theorem for a flow with finitely many singular points  (cf. \cite[Theorem~2.6.1]{nikolaev1999flows}) implies that the $\omega$-limit set $\partial$ is a limit circuit. 
By construction, the non-singular point $x_0$ is contained in the circuit $\partial$ with respect to $v$. 
\end{proof}

\subsection{Equivalence for the finite case}

In the last of this section, we show the following equivalence for flows with finitely many singular points.

\begin{proposition}\label{thm:finite_dense}
The following conditions are equivalent for a flow $v$ with finitely many singular points on a compact surface:
\\
$(1)$ $\overline{\mathop{\mathrm{Cl}}(v)} = \Omega(v)$.
\\
$(2)$ There are neither non-closed recurrent points, strict limit circuits, nor circuits with wandering holonomy.
\end{proposition}

To show the previous statement, we have the following existence of circuits.

\begin{lemma}\label{lem:bdry_circuit}
Let $v$ be a flow with finitely many singular points on a compact surface $S$, $x$ non-recurrent point,  $I$ an open transverse arc containing $x$, and $f_I$ the first return map on $I$. 
Suppose that $f_I$ is orientation-reversing and $f_I(I_-) = I_+$, where $I_-, I_+$ are the connected component of $I - \{ x \}$. 
Denote by $D$ the union of orbit arcs from points in $I_-$ to the image of $f_I$. 
The connected component of $\partial D - (I_- \sqcup I_+)$ containing $x$ is a circuit. 
\end{lemma}

\begin{proof}
By definition, the union $D$ is an open flow fox as in the left on Figure~\ref{nondir02}.
\begin{figure}
\begin{center}
\includegraphics[scale=0.15]{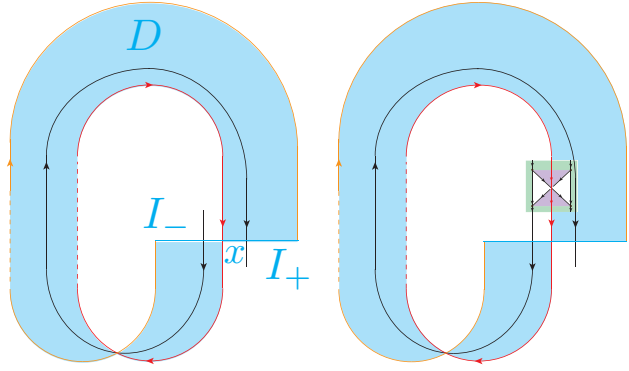}
\end{center}
\caption{Left, a trivial open flow fox and its neighborhood; right, the modified flow by replacing a trivial open flow box with a one-punctured M\"obius band as in Figure~\ref{twist01}.}
\label{nondir02}
\end{figure} 
Let $\gamma$ be the connected component of $\partial D - (I_- \sqcup I_+)$ containing $x$. 
Replace a small trivial open flow fox containing a point in $O^-(x)$ by a one-punctured M\"obius band as in Figure~\ref{twist01}, the resulting surface is a compact surface, and the resulting flow $w$ is a flow with finitely many singular points such that the first return map on $I_-$ by $w$ is repelling. 
\begin{figure}
\begin{center}
\includegraphics[scale=0.2]{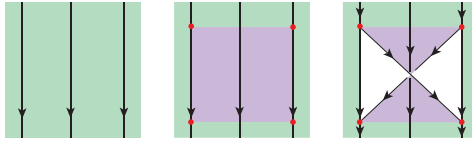}
\end{center}
\caption{Left, a trivial open flow fox $D$; middle, an open flow fox with exactly four singular points; right, a one-punctured M\"obius band obtained from the open flow fox with exactly four singular points by cutting an open disk, twisting the disk, pasting two disjoint closed intervals.}
\label{twist01}
\end{figure} 
Then $\gamma$ is the $\alpha$-limit set of a point in $I_-$ near $x$ with respect to $w$ and is a boundary component of an open annulus contained in $w(I_-)$. 
By the generalization of the Poincar\'e-Bendixson theorem for a flow with finitely many singular points, the $\alpha$-limit set $\gamma$ with respect to $w$ is a semi-repelling limit circuit with respect to $w$. 
 Since the replacement preserve $\gamma$, the subset $\gamma$ is also a circuit with respect to $v$.
\end{proof}

We have the following statement. 

\begin{lemma}\label{lem:wandering_h}
Every point in $\Omega(v) - \overline{\Cv}$ for a flow with finitely many singular points on a compact surface is contained in either a limit circuit or a circuit with wandering holonomy.
\end{lemma}

\begin{proof}
Let $v$ be a flow with finitely many singular points on a compact surface $S$. 
Fix a point $x \in \Omega(v) - \overline{\Cv}$ which is not contained in a limit circuit. 
Let $d$ be a Riemannian distance on $S$ and $T$ a small open transverse arc intersecting $x$ whose closure contains no singular points and is a closed transverse arc. 
Denote by $T_-$ and $T_+$ the connected components of $T - \{ x \}$. 
From $x \in \Omega(v) - \overline{\Cv}$, there is a sequence $(a_n)_{n \in \Z_{> 0}}$ of non-periodic points in $T$ with $O^+(a_n) \cap T \neq \emptyset$ and $O(a_n) \neq O(a_m)$ for any $n \neq m \in  \Z_{\geq 0}$ such that the sequencer $(a_n)_{n \in \Z_{\geq 0}}$ converges to $x$, $d(x, a_n) < 1/n$ and $d(x, b_n) < 1/n$ for any $n \in \Z_{> 0}$, where $b_n$ is the first return of $a_n$ for the first return map $f_T$ on the open transverse arc $T$. 
Renaming $T_+$ and $T_-$ and taking a subsequence of $(a_n)_{n \in \Z_{> 0}}$, we may assume that $a_n \in T_+$ for any $n \in \Z_{> 0}$ and that the sequences $(d(x, a_n))_{n \in \Z_{> 0}}$ and $(d(x, b_n))_{n \in \Z_{> 0}}$ of distances monotonically decrease. 

We claim that there are at most finitely many limit cycles each of which is the $\alpha$-limit or the $\omega$-limit sets of some point in $(a_n)_{n \in \Z_{> 0}}$. 
Indeed, assume that there are pairwise distinct infinitely many limit cycles $C_n$ each of which is the $\omega$-limit sets of some point in $(a_n)_{n \in \Z_{> 0}}$. 
By the finiteness of genus, taking a subsequence of $(C_n)_{n \in \Z_{> 0}}$ if necessary, we may assume that any limit cycles $C_n$ are parallel to each other. 
This means that the boundary $\partial U_x$ contains at most two of $C_n$. 
By construction, the boundary of the connected component $U_x$ of the complement $S - \bigsqcup_{n \in \Z_{> 0}} C_n$ contains $\bigsqcup_{n \in \Z_{> 0}} C_n$, which contradicts that $\partial U_x$ contains at most two of $C_n$. 
By symmetry of the time reversion, there are at most finitely many limit cycles each of which is the $\alpha$-limit sets of some point in $(a_n)_{n \in \Z_{> 0}}$. 

We claim that there are at most finitely many $\omega$-limit and $\alpha$-limit sets such that each of $\omega(a_n)$ and $\alpha(a_n)$ is one of them. 
Indeed, assume that there are pairwise distinct infinitely many $\omega$-limit sets of any points $a_n$. 
By the generalization of the Poincar\'e-Bendixson theorem for a flow with finitely many singular points, the $\omega$-limit and $\alpha$-limit sets of any points $a_n$ are singular points or limit circuits. 
By the finiteness of singular points and the previous claim, there are pairwise distinct infinitely many non-periodic limit circuits $C_n$ which are the $\omega$-limit sets of some points in $(a_n)_{n \in \Z_{> 0}}$.
Taking a subsequence of $(a_n)_{n \in \Z_{> 0}}$ and by time reversion if necessary, we may assume that $C_n = \omega(a_n)$. 
Since any collars for semi-attracting limit circuits $C_n$ are pairwise disjoint, the small associated collars are pairwise disjoint. 
By Lemma~\ref{lem:lc_transersal}, there are infinitely many closed transversals $\gamma_n$ with $\gamma_n \cap O^+(a_n) \neq \emptyset$ in the pairwise disjoint associated collars which are not contractible in their associated collars respectively.
By the finite existence of $\Sv$, taking a subsequence of $(a_n)_{n \in \Z_{> 0}}$, any closed transversals $\gamma_n$ are not contractible in $S$, because any contractible closed transversal bounds a disk containing a singular point by the Poincar\'e-Hopf theorem.
Then any connected components of the complement $S - \bigcup_{n} \gamma_n$ have at most finitely many boundary components. 
Denote by $U_x$ the connected component of the complement $S - \bigcup_{i} \gamma_i$ containing $x$. 
Then $U_x$ is invariant. 
Let $\mu_1, \ldots ,\mu_k$ be the boundary components of $U_x$ which are contained in $\bigcup_{i} \gamma_i$. 
Replacing $T$ with a small open subarc, we may assume that $T \cap \bigsqcup_{j=1}^k \mu_j = \emptyset$. 
Taking a subsequence of $(a_n)_{n \in \Z_{> 0}}$,  we may assume that $\gamma_i = \mu_i$ for any $j \in \{ 1,2, \ldots , k \}$. 
Then $\overline{U_x} \cap \bigsqcup_{n>k} \gamma_n = \emptyset$ and $\overline{U_x} = U_x \sqcup \bigsqcup_{j = 1}^k \gamma_j$. 
Fix any $n > k$. 
Since $\gamma_n \cap O^+(a_n) \neq \emptyset$, we obtain that $O^+(a_n) \setminus \overline{U_x} \neq \emptyset$. 
The invariance of $U_x$ implies that $a_n \notin U_x$, which contradicts $a_n \in T \subset U_x$. 
By symmetry of the time reversion, the claim is completed. 

Taking a subsequence of $(a_n)_{n \in \Z_{> 0}}$ if necessary, 
we may assume that there are closed invariant subsets $\alpha$ and $\omega$ in $S - \overline{T}$ such that $\alpha = \alpha(a_n)$ and $\omega = \omega(a_n)$ for any $n  \in \Z_{> 0}$. 
Since $x$ is not contained in any limit circuits, cutting limit circuits and collapsing the new boundary components into singletons if necessary, we may assume that each of $\alpha$ and $\omega$ is a singular point. 
Set $U_n$ the connected component of the complement $S - (\{ \alpha, \omega \} \sqcup \bigsqcup_{n\in \Z_{> 0}} O(a_n))$ containing $T_{(a_n, a_{n+1})}$. 
By the finite existence of genus and $\Sv$, taking a subsequence of $(a_n)_{n \in \Z_{> 0}}$ if necessary,  we may assume that any connected components $U_n$ are invariant open disks such that $\{ \alpha \} = \bigcup_{a \in U_n} \alpha(a)$ and $\{ \omega \} = \bigcup_{a \in U_n} \omega(a)$. 
By construction, the union $U_\infty := \bigsqcup_{n \in \Z_{>0}} U_n \sqcup O^-(a_{n+1})$ is an  invariant open disk and an invariant trivial flow box with $\{ \alpha \} = \bigcup_{a \in U_\infty} \alpha(a)$ and $\{ \omega \} = \bigcup_{a \in U_\infty} \omega(a)$ such that the sequence $(a_n)_{n \in \Z_{> 0}}$ in $T_+ \cap U_\infty$ converges to $x \in T_+ \cap \partial U_\infty$. 
Let $T_{(a,b)} \subset T$ be the open subinterval between $a$ and $b$ of $T$. 
Then there is a point $a_+ \in T_+ \cap U_\infty$ such that the subinterval $T_{(a_+,x)}$ is contained in $U_\infty$. 
Since $U_\infty$ is an invariant trivial flow box, the union $D := \bigsqcup_{a \in T_{(a_+,x)}} O^-(a) \subset U_\infty$ is a negative invariant open disk with $\{ \alpha \} = \bigcup_{a \in D} \alpha(a)$ and $T_{(a_+,x)} \sqcup \{ x \} \subset \partial D$. 
Because the sequence $(a_n)_{n \in \Z_{> 0}}$ in $T_+$ converges to $x$ and $b_n$ is the first return of $a_n$ for the first return map $f_T$ with $d(x, a_n) < 1/n$ and $d(x, b_n) < 1/n$ for any $n \in \Z_{> 0}$, 
there is a positive integer $N$ such that $a_n \in T_{(a_+,x)} \subset T_+ \cap U_\infty$ and $b_n \in T_-$ for any $n \in \Z_{> N}$. 
From $\lim_{m \to \infty} a_m = x = \lim_{m \to \infty} b_m$, by $a_n \in T_{(a_+,x)} \subset U_\infty$ and  $b_n \in T_- \cap O^+(a_n) \subset U_\infty$ for any $n \in \Z_{> N}$, since $U_\infty$ is an invariant trivial flow box and $T$ is an open transverse interval, there is a point $a_- \in T_-$ such that $T_{(x, a_-)} \subset U_\infty$ with $f_T(T_{(a_+,x)}) = T_{(x, a_-)}$ such that the restriction $f_T|_{T_{(a_+,a_-)}} \colon T_{(a_+,x)} \to T_{(x, a_-)}$ is orientation-reversing. 
Applying Lemma~\ref{lem:bdry_circuit} to $I := T_{(a_+, a_-)}$, there is a circuit $\nu$ in $\partial U_\infty$ with wandering holonomy such that $x \in \nu$.  
\end{proof}

We have the following classification of a non-recurrent point under the non-existence of non-closed recurrent orbits. 

\begin{lemma}\label{lem:wandering_pt}
Let $v$ be a flow with finitely many singular points on a compact surface $S$. 
Suppose that there are no non-closed recurrent points. 
Then one of the following statements holds exclusively for a non-recurrent point $x$: 
\\
{\rm(1)} The orbit $O$ is non-wandering and satisfies either {\rm(1.1)}, {\rm(1.2)}, or {\rm(1.3)} exclusively: 
\\
{\rm(1.1)} There is a strict limit circuit in $\Omega(v)$ containing $x$. 
\\
{\rm(1.2)} There is a circuit with wandering holonomy containing $x \in \Omega(v)$.
\\
{\rm(1.3)} There is a circuit in $\overline{\Pv} - \Pv \subseteq \Omega(v)$ containing $x$. 
\\
{\rm(2)} The point $x$ is wandering {\rm(i.e.} $x \notin \Omega(v)${\rm)}. 
\end{lemma}

\begin{proof}
By Gutierrez's smoothing theorem~\cite{gutierrez1978structural}, we may assume that $v$ is smooth. 
Since $\Pv \subseteq \Omega(v)$, the closedness of $\Omega(v)$ implies that $\overline{\Pv} \subseteq \Omega(v)$. 
Fix a non-recurrent point $x$. 
We may assume that there is no circuit $\gamma \in \overline{\Pv}$ containing $x$. 
The closedness of $\Sv$ implies that there is no circuit $\gamma \in \overline{\Cv}$ containing $x$. 
Lemma~\ref{lem:circuit} implies $x \notin \overline{\Pv}$ and so $x \notin \overline{\Cv}$. 
If $x$ is contained in a circuit with wandering holonomy, then $x \in \Omega(v)$. 
Thus we may assume that $x$ is not contained in a circuit with wandering holonomy. 

We claim that every limit circuit containing $x$ is a strict limit circuit with $x \in \Omega(v)$. 
Indeed, suppose that $x$ is contained in a limit circuit $\gamma$. 
Then $x \in \Omega(v) \cap (\gamma \setminus \overline{\Cv}) = (\Omega(v) - \overline{\Cv}) \cap \gamma$. 
If $\gamma$ is one-sided, then it is a strict limit circuit. 
Thus we may assume that $\gamma$ is not one-sided. 
Let $\A_+$ be a small associated collar of the limit circuit $\gamma$ such that $\gamma$ is either semi-attracting or semi-repelling with respect to $\A_+$. 
Since $\gamma \setminus \overline{\Cv} \neq \emptyset$, there is a small open transverse arc $I$ containing a non-recurrent point $x_0 \in \gamma \setminus \overline{\Cv}$. 
Denote by $I_-$ the connected component of $I - \{ x_0 \}$ which does not intersect $\A_+$. 
By $\gamma \setminus \overline{\Pv} \neq \emptyset$, the first return map on $I_-$ has no fixed points, and so the circuit $\gamma$ is a strict limit circuit. 

Thus we may assume that $x$ is contained in neither a limit circuit nor a circuit with wandering holonomy. 
We claim that $x$ is wandering. 
Indeed, $x$ is non-wandering. 
Since the point $x \in \Omega(v) - \overline{\Pv}$ is not contained in any limit circuits, Lemma~\ref{lem:wandering_h} implies that there is a circuit with wandering holonomy containing $x$, which contradicts that $x$ is not contained in any circuits with wandering holonomy.
\end{proof}

%
%
%
%
%

 In the spherical case, we have the following classification of non-closed orbits. 
 
\begin{corollary}\label{cor:wandering_pt}
Let $v$ be a flow with finitely many singular points on a compact surface contained in a sphere. 
Then one of the following statements holds exclusively for a non-closed point $x$: 
\\
{\rm(1)} The orbit $O$ is non-wandering and satisfies either {\rm(1.1)} or {\rm(1.2)} exclusively: 
\\
{\rm(1.1)} The point $x$ is contained in a strict limit circuit and $x \in \Omega(v)$. 
\\
{\rm(1.2)} There is a circuit $\gamma \in \overline{\Pv} - \Pv \subseteq \Omega(v)$ containing $x$. 
\\
{\rm(2)} The point $x$ is wandering {\rm(i.e.} $x \notin \Omega(v)${\rm)}. 
\end{corollary}
 
\begin{proof}
Since any flows on compact surfaces contained in a sphere have no non-closed recurrent orbits, any non-recurrent orbits are non-closed.
By definition of wandering holonomy, the orientability of the sphere implies the non-existence of circuits with wandering holonomy.  
Lemma~\ref{lem:wandering_pt} implies the assertion. 
\end{proof}

We have the following sufficient condition for the denseness of closed orbits with respect to 
a flow with finitely many singular points on a surface. 

\begin{lemma}\label{lemp:sufficient_dense}
Let $v$ be a flow with finitely many singular points on a compact surface $S$ without non-closed recurrent points, strict limit circuits, nor circuits with wandering holonomy. 
Then $\overline{\mathop{\mathrm{Cl}}(v)} = \Omega(v)$. 
\end{lemma}

\begin{proof}
The non-existence of non-closed recurrent points implies that $S = \Cv \sqcup \mathrm{P}(v)$. 
Fix a point $x \in \mathrm{P}(v)$. 
From Lemma~\ref{lem:wandering_pt}, either $x \notin \Omega(v)$ or there is a circuit $\gamma \in \overline{\Pv} \subseteq \Omega(v)$ containing $x$.  
Thus either $x \in \mathrm{P}(v) \setminus \Omega(v)$ or $x \in \overline{\Pv} \subseteq \Omega(v)$. 
By $\overline{\Pv} \subseteq \Omega(v)$, we have $\mathrm{P}(v) = (\mathrm{P}(v) \setminus \Omega(v)) \sqcup ((\mathrm{P}(v) \cap \overline{\Pv})$ and so $\mathrm{P}(v) \cap \Omega(v) = \mathrm{P}(v) \cap \overline{\Pv}$. 
Since $S - \mathrm{P}(v) = \Cv \subseteq \Omega(v)$, we obtain $\Omega(v) = S \cap \Omega(v) = (\Cv \sqcup \mathrm{P}(v)) \cap \Omega(v) = (\Cv \cap \Omega(v)) \sqcup (\mathrm{P}(v) \cap \Omega(v)) = \Cv \sqcup (\mathrm{P}(v) \cap \overline{\Pv}) \subseteq \overline{\Cv} \subseteq \Omega(v)$. 
This means that $\overline{\mathop{\mathrm{Cl}}(v)} = \Omega(v)$. 
\end{proof}

Proposition~\ref{prop:necessary_dense} and Lemma~\ref{lemp:sufficient_dense} imply Proposition~\ref{thm:finite_dense}. 

%

\section{Characterizations of density of closed orbits of flows with finitely many singular points}

\subsection{Topological characterization for the case with finitely many connected components of the singular point set}

To generalize Proposition~\ref{thm:finite_dense} into a characterization of the denseness of closed orbits for a flow with finitely many connected components of the singular point set, we introduce some concepts (see \cite{yokoyama2017decompositions} for details of these constructions).

\subsubsection{Blow-downs for singular points}
We define blow-downs for the singular point set of a surface and their flows as follows (see Figure~\ref{Fig:blowdown}):
\begin{figure}
\[
\xymatrix@=18pt{
S \ar@{}[d]|{\bigcup} & & S_{\mathrm{mc}}\ar[ll]_{\pi_{\mathrm{mc}}} \ar[rr]^q \ar@{}[d]|{\bigcup} & &  S_{\mathrm{col}} \ar@{}[d]|{\bigcup} \\
S - \mathop{\mathrm{Sing}}(v) & &  S_{\mathrm{mc}} - \mathop{\mathrm{Sing}}(v_{\mathrm{mc}}) \ar@{=}[ll]_{\pi_{\mathrm{mc}}|}
 \ar@{=}[rr]^{q|} & &  S_{\mathrm{col}} - \mathop{\mathrm{Sing}}(v_{\mathrm{col}})}
\]
\caption{Canonical quotient mappings induced by the metric completion and the collapse.}
\label{Fig:blowdown}
\end{figure}
Let $v$ be a flow on a surface $S$ whose singular point set has at most finitely many connected components. 
Since the singular point set $\mathop{\mathrm{Sing}}(v)$ is closed, the complement $S - \mathop{\mathrm{Sing}}(v)$ is open and so is a surface. 
In particular, the set difference $S - (\mathop{\mathrm{Sing}}(v) \cup \partial S)$ is an open surface without boundary. 
Fix a Riemannian metric on $S$ such that $\mathop{\mathrm{Sing}}(v)$ is bounded. 
Denote by $S_{\mathrm{mc}}$ the metric completion of the complement $S - \mathop{\mathrm{Sing}}(v)$. 
Identifying the union $\partial$ of new boundary components with the new singular points, define a flow $v_{\mathrm{mc}}$ on $S_{\mathrm{mc}}$ such that $O_v(x) = O_{v_{\mathrm{mc}}}(x)$ for any point $x \in S - \mathop{\mathrm{Sing}}(v) = S_{\mathrm{mc}} - \mathop{\mathrm{Sing}}(v_{\mathrm{mc}})$ up to topological equivalence. 
Then $\partial = \mathop{\mathrm{Sing}}(v_{\mathrm{mc}})$ and so $S - \mathop{\mathrm{Sing}}(v) = S_{\mathrm{mc}} - \mathop{\mathrm{Sing}}(v_{\mathrm{mc}})$. 
By \cite[Theorem~3]{reeb1952certaines}, each connected component of the open surface $S - (\mathop{\mathrm{Sing}}(v) \cup \partial S)$ 
 without boundary is homeomorphic to the resulting surface from a closed surface by removing a closed totally disconnected subset. 
Therefore, collapsing each connected component of $\mathop{\mathrm{Sing}}(v_{\mathrm{mc}})$ into a singular point (as in Figure~\ref{blowup}), we obtain the resulting flow $v_{\mathrm{col}}$ with totally disconnected singular points, called the blow-down flow of $v$, on the resulting surface $S_{\mathrm{col}}$, called the blow-down surface, up to topological equivalence. 
\begin{figure}
\begin{center}
\includegraphics[scale=0.1]{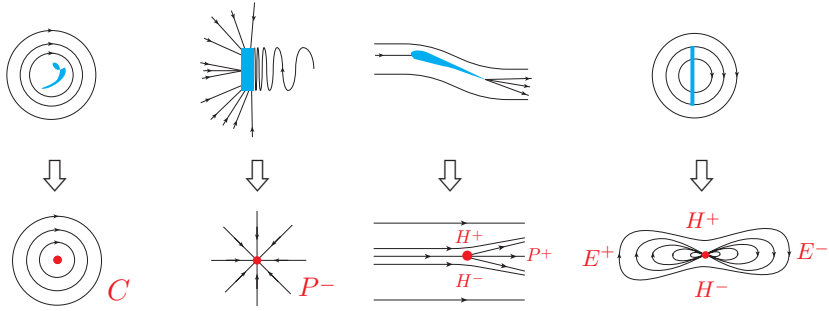}
\end{center}
\caption{Blow-downs of finitely many connected components of the singular point set.}
\label{blowup}
\end{figure}
Then $O_v(x) = O_{v_{\mathrm{mc}}}(x) = O_{v_{\mathrm{col}}}(x)$ for any point $x \in S - \mathop{\mathrm{Sing}}(v) = S_{\mathrm{mc}} - \mathop{\mathrm{Sing}}(v_{\mathrm{mc}}) = S_{\mathrm{col}} -  \mathop{\mathrm{Sing}}(v_{\mathrm{col}})$. 
Notice that $S_{\mathrm{mc}}$ and $S_{\mathrm{col}}$ may have infinitely many connected components but that each connected component of $S_{\mathrm{col}}$ is a compact surface.

\subsubsection{Concepts related to blow-downs}
We recall some concepts to state the characterization of the density of closed orbits in the non-wandering set. 
%
%
An invariant subset $\gamma$ is a blow-up of a circuit with wandering holonomy if the image $q(\pi_{\mathrm{mc}}^{-1}(\gamma))$ is a circuit with wandering holonomy for $v_{\mathrm{col}}$. 
%
A closed connected invariant subset is a non-trivial quasi-circuit if it is a boundary component of an open annulus, contains a non-recurrent point, and consists of non-recurrent points and singular points. 
A non-trivial quasi-circuit $\gamma$ is a quasi-semi-attracting (resp. quasi-semi-repelling) limit quasi-circuit (with respect to a small collar $\A$) if there is a point $x \in \A$ with $O^+(x) \subset \A$ (resp. $O^-(x) \subset \A$) such that $\omega(x)  = \gamma$ (resp. $\alpha(x)  = \gamma$). 
A non-trivial quasi-circuit is a limit quasi-circuit (with respect to a small collar $\A$) if it is a quasi-semi-attracting/quasi-semi-repelling limit quasi-circuit with respect to $\A$. 
%
%
%
By construction, we have the following observation. 

\begin{lemma}\label{lem:blow_down}
The following statements hold for a flow $v$ with finitely many connected components of the singular point set on a compact surface $S$: 
\\
{\rm(1)} $\overline{\mathop{\mathrm{Cl}}(v)} = \Omega(v)$ if and only if $\overline{\mathop{\mathrm{Cl}}(v_{\mathrm{col}})} = \Omega(v_{\mathrm{col}})$.
\\
{\rm(2)} $\mathrm{R}(v) = \emptyset$ if and only if $\mathrm{R}(v_{\mathrm{col}}) = \emptyset$. 
\\
{\rm(3)} The $\omega$-limit set $\omega_v(x)$ of a non-singular point $x$ on $S$ is a limit quasi-circuit if and only if the $\omega$-limit set $\omega_{v_{\mathrm{col}}}(x)$ is a non-periodic limit circuit. 
\\
{\rm(4)} 
There are no strict limit quasi-circuits for $v$ if and only if there are no strict limit circuits for $v_{\mathrm{col}}$. 
\\
{\rm(5)} There are no blow-ups of circuits with wandering holonomy for $v$ if and only if there are no circuits with wandering holonomy for $v_{\mathrm{col}}$. 
\end{lemma}

Proposition~\ref{thm:finite_dense} and Lemma~\ref{lem:blow_down} imply the following statement.

\begin{corollary}\label{thm:main03}
Let $v$ be a flow with finitely many connected components of the singular point set  on a compact surface $S$.
The following conditions are equivalent:
\\
$(1)$ $\overline{\mathop{\mathrm{Cl}}(v)} = \Omega(v)$.
\\
$(2)$
There are neither non-closed recurrent points, strict limit quasi-circuits, nor blow-ups of circuits with wandering holonomy.
\\
$(3)$ Each orbit is proper, and there are neither strict limit quasi-circuits nor blow-ups of circuits with wandering holonomy.
\\
$(4)$ The orbit space $S/v$ is $T_0$, and there are neither strict limit quasi-circuits nor blow-ups of circuits with wandering holonomy.
\end{corollary}

The previous theorem implies Corollary~\ref{main:a}. 
Note the finiteness of connected components of the singular point set is necessary.
In other words, $\overline{\mathop{\mathrm{Cl}}(v)} \subsetneq \Omega(v)$ for a flow $v$ on a compact surface in general.
Indeed, there is a flow $v$ on a compact surface $S$ with $\mathrm{R}(v) = \emptyset$ and $\overline{\mathop{\mathrm{Cl}}(v)} \subsetneq \Omega(v)$ but without strict limit circuits nor circuits with wandering holonomy (e.g. Lemma~\ref{lem:counter_exam00}).

\subsection{Density of closed orbits in the non-wandering set and correspondence for flows on non-compact surfaces}

Since some kinds of fluid phenomena are described as flows on non-compact domains, to describe such phenomena, we generalize our results into the case for non-compact surfaces. 
Therefore, we introduce end completions of surfaces with flows. 

\subsubsection{End completions of surfaces with finite genus} 
Recall the end completion, which is introduced by Freundenthal \cite{Freudenthal1931end}, as follows. 
Consider the direct system $\{K_\lambda\}$ of compact subsets of a topological space $X$ and inclusion maps such that the interiors of $K_\lambda$ cover $X$.  
There is a corresponding inverse system $\{ \pi_0( X - K_\lambda ) \}$, where $\pi_0(Y)$ denotes the set of connected components of a space $Y$. 
Then the set of ends of $X$ is defined to be the inverse limit of this inverse system. 
Notice that $X$ has one end $x_{\mathcal{U}}$ for each sequence $\mathcal{U} := (U_i)_{i \in \mathbb{Z}_{>0}}$ with $U_i \supseteq U_{i+1}$ such that $U_i$ is a connected component of $X - K_{\lambda_i}$ for some $\lambda_i$. 
Considering the disjoint union $X_{\mathrm{end}}$ of $X$ and  $\{ \pi_0( X - K_\lambda ) \}$ as a set, a subset $V$ of the union $X_{\mathrm{end}}$ is an open \nbd of an end $x_{\mathcal{U}}$ if there is some $i \in \mathbb{Z}_{>0}$ such that $U_i \subseteq V$. 
Then the resulting topological space $X_{\mathrm{end}}$ is called the end completion (or end compactification) of $X$. 
Note that the end completion is not compact in general. 
%
From \cite[Theorem~3]{richards1963classification}, any connected surfaces of finite genus are  homeomorphic to the resulting surfaces from closed surfaces by removing closed totally disconnected subsets. 
Therefore the end compactification $S_{\mathrm{end}}$ of a connected surface $S$ of finite genus is a closed surface. 

For a flow $v$ on a surface $S$ of finite genus, considering ends to be singular points, we obtain the resulting flow $v_{\mathrm{end}}$ on a surface $S_{\mathrm{end}}$ which is a union of closed surfaces. 
%
We have the following observation. 

\begin{lemma}\label{lem:blow_down02}
The following statements hold for a flow $v$ on a surface $S$ with finite genus: 
\\
{\rm(1)} Each orbit with respect to $v$ is proper if and only if each orbit with respect to $v_{\mathrm{end}}$ is proper. 
\\
{\rm(2)} The orbit space $S/v$ is $T_0$ if and only if the orbit space $S_{\mathrm{end}}/v_{\mathrm{end}}$ is $T_0$.
\end{lemma}

\begin{proof}
\cite[Theorem~3.3]{yokoyama2019properness} implies that the assertions {\rm(1)} and {\rm(2)} are equivalent. 
Note that the end completion for $S$ adds finitely many singular points with respect to the resulting flow $v_{\mathrm{end}}$. 
If $S_{\mathrm{end}}/v_{\mathrm{end}}$ is $T_0$, then so is the subspace $S/v_{\mathrm{end}} = S/v$. 
Conversely, suppose that $S/v$ is $T_0$.
From \cite[Theorem~3]{richards1963classification}, any connected surfaces of finite genus are  homeomorphic to the resulting surfaces from closed surfaces by removing closed totally disconnected subsets. 
Therefore the set of end points is closed and so $S$ is open in $S_{\mathrm{end}}$. 
By definition of $T_0$ axiom, the $T_0$ axiom for $S/v$ implies one for the end completion $S_{\mathrm{end}}/v_{\mathrm{end}}$. 
\end{proof}

\subsubsection{Virtually strict limit quasi-circuits and virtually quasi-circuits}
An invariant subset is a virtually limit (resp. strict limit) quasi-circuit if it is the resulting subset from a limit (resp. strict limit) quasi-circuit on $S_{\mathrm{end}}$ with respect to $v_{\mathrm{end}}$ by removing all the ends. 
An invariant subset $\mu$ is a virtual blow-up of a circuit with wandering holonomy if there is a blow-up of a circuit $\gamma$ with wandering holonomy on $S_{\mathrm{end}}$ with respect to $v_{\mathrm{end}}$ such that the resulting subset from $\gamma$ by removing all the ends is $\mu$. 

Corollary~\ref{thm:main03} and Lemma~\ref{lem:blow_down02} imply the following characterization of the denseness of closed orbits
for a flow with finitely many connected components of the singular point set on a surface $S$ with finite genus and finite ends.

\begin{corollary}\label{thm:main}
The following conditions are equivalent for a flow $v$ with finitely many connected components of the  singular point set on a surface $S$ with finite genus and finite ends:
\\
$(1)$ $\overline{\mathop{\mathrm{Cl}}(v)} = \Omega(v)$.
\\
$(2)$
There are neither non-closed recurrent points, virtually strict limit quasi-circuits, nor virtual blow-up of circuits with wandering holonomy.
\\
$(3)$ Every orbit is proper, and there are neither virtually strict limit quasi-circuits nor virtual blow-up of circuits with wandering holonomy.
\\
$(4)$ The orbit space $S/v$ is $T_0$, and there are neither virtually strict limit quasi-circuits nor virtual blow-up of circuits with wandering holonomy.
\end{corollary}


\section{Characterization of non-closed behaviors}

We have the following characterization of orbits. 

\begin{theorem}\label{main:rec_wandering}
Let $v$ be a flow with finitely many connected components of the singular point set on a surface $S$ with finite genus and finite ends. 
Then one of the following three statements holds for an orbit $O$ exclusively: 
\\
{\rm(1)} The orbit $O$ is recurrent {\rm (i.e.} $O \subseteq \Cv \sqcup \mathrm{R}(v)${\rm)}. 
\\
{\rm(2)} The orbit $O$  is not recurrent but non-wandering and satisfies $\overline{O} - O = \alpha(O) \cup \omega(O) \subseteq \Sv$ and either {\rm(2.1)}, {\rm(2.2)}, {\rm(2.3)}, or {\rm(2.4)}: 
\\
{\rm(2.1)} There is a virtually strict limit quasi-circuit in $\Omega(v)$ containing $O \subset \mathrm{int} \mathrm{P}(v)$. 
\\
{\rm(2.2)} There is the virtual blow-up of a circuit with wandering holonomy containing $O \subset \mathrm{int} \mathrm{P}(v)$. 
\\
{\rm(2.3)} There is the virtual blow-up of a circuit in $\overline{\Pv} - \Pv \subseteq \Omega(v)$ containing $O$. 
\\
{\rm(2.4)} $O \subset \overline{\mathrm{R}(v)} - \mathrm{R}(v)$. 
\\
{\rm(3)} The orbit $O$ is wandering {\rm(i.e.} $O \cap \Omega(v) = \emptyset${\rm)} and $O \subseteq \mathrm{int} \mathrm{P}(v)$. 
\end{theorem}

\begin{proof}
Fix an orbit $O$ on $S$. 
If $O$ is recurrent, then the assertion {\rm(1)} holds. 
Thus we may assume that $O$ is not recurrent (i.e. $O \cap \mathrm{R}(v) = \emptyset$). 
If $O$ is wandering, then the assertion {\rm(3)} holds. 
Thus we may assume that $O$ is not-wandering. 

We claim that we may assume that $O \cap \overline{\mathrm{R}(v)} = \emptyset$. 
Indeed, suppose that $O \subset \overline{\mathrm{R}(v)}$.
Then the assertion {\rm(2.4)} holds. 
Since $O$ is non-recurrent and so is non-closed proper, we have $\overline{O} - O = \alpha(O) \cup \omega(O)$ (cf. \cite[Lemma 4.1]{yokoyama2021refine}). 
By the finite existence of quasi-minimal sets of flows on compact surfaces, the surface $S$ has at most finitely many quasi-minimal sets. 
Therefore any orbit in $\overline{R(v)}$ is contained in a quasi-minimal set which is either the $\alpha$-limit set or the $\omega$-limit set of a point. 
\cite[Theorem 3.15]{yokoyama2021poincare} implies that $\alpha(O) \cup \omega(O) \subseteq \Sv$. 

Therefore it suffices to show the existence of a virtually strict limit quasi-circuit in $\Omega(v)$ containing $O$ or the virtual blow-up of a circuit which is contained in $\overline{\Pv} - \Pv \subseteq \Omega(v)$ and contains $O$. 
Taking the end completion $S_{\mathrm{end}}$ of $S$, let $v_{\mathrm{end}}$ be the resulting flow on the $S_{\mathrm{end}}$. 
By definitions of virtually strict limit quasi-circuit and the virtual blow-up of a circuit, it suffices to show the existence of a strict limit quasi-circuit in $\Omega(v_{\mathrm{end}})$ containing $O$ or the blow-up of a circuit which is contained in $\overline{\mathop{\mathrm{Per}}(v)} - \mathop{\mathrm{Per}}(v)_{\mathrm{end}} \subseteq \Omega(v_{\mathrm{end}})$ and contains $O$. 
Blow-downing $S_{\mathrm{end}}$, let $w$ be the resulting flow on the resulting surface $T$. 
By definitions of strict limit quasi-circuit and the blow-up of a circuit, it suffices to show the existence of a strict limit circuit in $\Omega(w)$ containing $O$ or a circuit which is contained in $\overline{\mathop{\mathrm{Per}}(w)} - \mathop{\mathrm{Per}}(w) \subseteq \Omega(w)$ and contains $O$. 
Then the singular point set $\mathop{\mathrm{Per}}(w)$ consists of finitely many points and the surface $T$ is compact. 
Moreover, we have $O \cap \overline{\mathrm{R}(w)} = \emptyset$. 
For any quasi-minimal set, there is a closed transversal intersecting it. 
Cutting such closed transversals and adding pairs of sinks and sources, the resulting surface $T'$ is a compact surface and the resulting flow $w'$ has at most finitely many singular points. 
Then $O$ is non-recurrent with respect to $w'$. 

We claim that we may assume that $O$ is non-wandering with respect to $w'$. 
Indeed, assume that $O$ is wandering with respect to $w'$. 
Since $O$ is wandering with respect to $w'$, there there are a point $x \in O$, a closed transversal $T$ intersecting a quasi-minimal set $\mathcal{M}$, a closed transverse arc $I$ with $x \in \partial I$ and $I \cap \mathop{\mathrm{Per}}(w) = \emptyset$, sequences $(x_n)_{n \in \Z_{>0}}$ and $(z_n)_{n \in \Z_{>0}}$ in $I$ monotonically converging to $x$, and a sequence $(y_n)_{n \in \Z_{>0}}$ on $T$ such that either $x_n \in O^+(y_n)$ and $z_n \in O^-(y_n)$ or $x_n \in O^-(y_n)$ and $z_n \in O^+(y_n)$. 
Fix an order $\leq$ on $I$ such that $x$ is the maximal element. 
Then $x_n < x_{n+1}$ and $z_n < z_{n+1}$ for any $n \in \Z_{>0}$. 
By time reversion if necessary, we may assume that $x_n \in O^+(y_n)$, $z_n \in O^-(y_n)$ and $z_1 < x_1$. 
Replacing $I$ with a subinterval, we may assume that $z_1 \in \partial I$. 
From $I \cap \mathop{\mathrm{Per}}(w) = \emptyset$, we obtain $z_1 < f_v(z_1)$. 
Let $f_v$ be the first return map on $I$ induced by $v$, $C_{a,b} \subset O^+(x)$ the orbit arc from $a$ to $b$, and $I_{a,b} \subset I$ the subinterval from $a$ to $b$ of $I$. 
%
For any $n \in \Z_{>0}$, the the union $\gamma_n : = C_{z_n, f_v(z_n)} \cup I_{z_n, z_{n+1}} \cup C_{z_{n+1}, f_v(z_{n+1})} \cup I_{f_v(z_{n}), f_v(z_{n+1})}$ is the immersed image of a circle and denote by $D_n$ the connected component of $S - \gamma_n$ containing the interior of $C_{z, f_v(z)}$ for any point $z \in I_{z_n, z_{n+1}} - \partial I_{z_n, z_{n+1}}$. 
By finite existence of genus and singular points, by renumbering of $n$, we may assume that $f_v$ is orientation-preserving and that any connected component $D_n$ is an open disk for any $n \in \Z_{>0}$. 
Then there is a continuous function $t_1 \colon I_{z_1, z_2} \to \R$ such that $f_v(z) = v(t_1(z),z) \in f_v(I_{z_1, z_2})$. 
For any $n \in \Z_{>0}$, by the existence of trivial flow boxes of $C_{z_{n}, f_v(z_{n})}$, since $D_n$ is a trivial flow box, there is a continuous function $t_n \colon I_{z_1, z_n} \to \R$ such that $f_v(z) = v(t_n(z),z) \in f_v(I_{z_1, z_n})$. 
Since $I = I_{z_1, x}$, there is a continuous function $t_\infty \colon I \to \R$ such that $f_v(z) = v(t_\infty(z),z) \in \bigcup_{n \in \Z_{>0}} f_v(I_{z_1, z_n}) \subseteq I$. 
If $\overline{t_\infty(I)} \subset \mathrm{int} I$, then the contraction principle implies the existence of a fixed point of $f_v$ and so that there is a periodic point intersecting $I$, which contradicts $I \cap \mathop{\mathrm{Per}}(w) = \emptyset$. 
Thus $x \in \overline{t_\infty(I)}$.  
Therefore the union $A := \bigcup_{n \in \Z_{>0}} D_n \cup C_{z_{n}, f_v(z_{n+1})} \cup \mathrm{int} \, t_\infty(I)$ is an open annulus. 
Let $\partial$ be the boundary component of $A$ containing $x$. 
For any $z \in A$, we obtain $\omega(z) = \partial$. 
This means that $x$ is non-wandering, which contradicts that $x$ is wandering.

Applying Lemma~\ref{lem:wandering_pt}, we have one of the following statements: {\rm(1)} There is a strict limit circuit in $\Omega(w')$ containing $O$; {\rm(2)} There is a circuit with wandering holonomy containing $O \subseteq \Omega(w')$; {\rm(3)} There is a circuit in $\overline{\mathop{\mathrm{Per}}(w')} - \mathop{\mathrm{Per}}(w') \subseteq \Omega(w')$ containing $x$. 
If the assertions {\rm(1)} (resp. {\rm(2)}, {\rm(3)}) holds, then there is a strict limit circuit (resp. a circuit with wandering holonomy, a circuit in $\overline{\mathop{\mathrm{Per}}(w)} - \mathop{\mathrm{Per}}(w) \subseteq \Omega(w)$) in $\Omega(w)$ containing $O$. 
\end{proof}

Theorem~\ref{main:rec_wandering_pt} follows from Theorem~\ref{main:rec_wandering} via the end completion. 
We have the following statement, which is a generalization of the Poincar\'e recurrence theorem for flows on surfaces.

\begin{theorem}\label{thm:rec_th}
A flow with finitely many connected components of the singular point set on a surface with finite genus and finite ends has neither virtually strict limit quasi-circuits nor the virtual blow-ups of circuits with wandering holonomy if and only if the set of recurrent points is dense in the non-wandering set.
In any case, the set of wandering points corresponds to the interior of the set of non-recurrent points. 
\end{theorem}

\begin{proof}
Let $v$ be a flow with finitely many connected components of the singular point set on a surface $S$ with finite genus and finite ends. 
Notice that $S = \Cv \sqcup \mathrm{P}(v) \sqcup \mathrm{R}(v)$. 
Suppose that there are neither virtually strict limit quasi-circuits nor the virtual blow-ups of circuits with wandering holonomy. 
Theorem~\ref{main:rec_wandering} implies that the set of wandering points corresponds to $\Sv \cup \overline{\Pv} \cup \overline{\mathrm{R}(v)} = S - \overline{\Cv \sqcup \mathrm{R}(v)}= \mathrm{int} \mathrm{P}(v)$ because of the closedness of $\Sv$. 
Therefore the non-wandering set $\Omega(v) = S - \mathrm{int} \mathrm{P}(v) = \overline{\Cv \sqcup \mathrm{R}(v)}$. 
This means that the closure of the set of non-recurrent points corresponds to the non-wandering set.

Conversely, suppose that $\Omega(v) = \overline{\Cv \sqcup \mathrm{R}(v)}$. 
From Theorem~\ref{main:rec_wandering}, there are neither virtually strict limit quasi-circuits nor the virtual blow-ups of circuits with wandering holonomy. 
\end{proof}

Theorem~\ref{main_cor:rec_th} follows from Theorem~\ref{thm:rec_th}. 
The characterization of minimality for flows on connected surfaces follows from  Theorem~\ref{main:rec_wandering_pt} as follows. 


\begin{proof}[Proof of Theorem~\ref{main:minimal}]
Let $v$ be a flow on a connected surface with finite genus and finite ends. 
Suppose that $v$ is minimal. 
Then $v$ is non-wandering and $S = \mathrm{LD}(v)$. 
Therefore $\Sv = \emptyset$. 
Fix any orbit $O$. 
From \cite[Theorem~VI]{cherry1937topological}, the closure $\overline{O}$ contains uncountably many non-closed recurrent orbits and so $\emptyset \neq \overline{O} - O = S -  O = \mathrm{LD}(v) - O \not\subseteq \Sv$. 

Conversely, suppose that $v$ is non-wandering and the set difference $\overline{O} - O$ for any orbit $O$ is not contained in the singular point set.
Since the set difference $\overline{O} - O$ for any closed orbit $O$ is empty, there are no closed orbits. 
In particular, the singular point set is empty. 
From Theorem~\ref{main:rec_wandering_pt}, any orbits are non-closed recurrent. 
Lemma~\ref{lem:top23} implies that $\mathrm{E}(v) = \emptyset$ and so that $S = \mathrm{LD}(v)$. 
The local density implies that the closure $\overline{O(x)}$ for any point $x \in \mathrm{LD}(v) = S$ is a \nbd of $x$. 
Fix an orbit $O$. 
By \cite[Proposition~2.2]{yokoyama2016topological}, the union $\{ x \in S \mid \overline{O(x)} = \overline{O} \} = \overline{O}$ is a minimal set. 
Therefore the closure $\overline{O}$ is a \nbd of every point of $\overline{O}$ and so is open. 
Since any nonempty open and closed subset of the connected surface $S$ is the whole surface $S$, we obtain $\overline{O} = S$. 
\end{proof}

\section{Characterization of finiteness of the non-wandering set and correspondence to the closed point set}
 
We state the following statement. 

\begin{lemma}\label{lem:circuit_w_wh_f}
Let $v$ be a flow with finitely many singular points on a compact surface. 
If each recurrent orbit is a closed orbit, then any circuits with wandering holonomy contain at most finitely many non-wandering orbits. 
\end{lemma}

\begin{proof}
Let $v$ be a flow with finitely many singular points on a compact surface with $\mathrm{R}(v) = \emptyset$. 
Since any point whose $\omega$-limit or $\alpha$-limit set is a limit circuit is wandering, the generalization of the Poincar\'e-Bendixson theorem for a flow with finitely many singular points implies that the $\omega$-limit and $\alpha$-limit sets of any non-recurrent non-wandering points are singular points. 
Assume that there is a circuit $\gamma$ with wandering holonomy that contains infinitely many non-wandering orbits $O_n$ ($n \in \Z_{> 0}$).
Then the $\omega$-limit and $\alpha$-limit sets of $O_n$ are singular points for any $n \in \Z_{> 0}$. 
By the finite existence of singular points, there is a singular point $\alpha$ (resp. $\omega$) which is the $\omega$-limit (resp. $\alpha$-limit) set of $O_n$ for infinitely many $n \in \Z_{> 0}$. 
Taking a subsequence of $(O_n)_{n \in \Z_{> 0}}$, we may assume that $\omega(O_n) = \omega$ and $ \alpha(O_n) = \alpha$ for any $n \in \Z_{> 0}$. 
Denote by $D_n$ the connected component of $S - (\{ \alpha, \omega \} \sqcup \bigsqcup_{n \in \Z_{>0}} O_n)$ with $\partial D_n = \{ \alpha, \omega \} \sqcup O_n \sqcup O_{n+1}$ for any $n \in \Z_{> 0}$. 
By the finite existence of genus and singular points, taking a subsequence of $(O_n)_{n \in \Z_{> 0}}$, we may assume that any $D_n$ is an invariant open disk. 
Then any $D_i$ intersects no periodic points and so consists of non-recurrent points.
This implies that any $O_{n+1}$ for any $n \in \Z_{> 0}$ has no non-orientable holonomy and so is not contained in a circuit with wandering holonomy, which contradicts the definition of $O_{n+1}$. 
\end{proof}

We demonstrate Corollary~\ref{main:c} as follows.

\begin{proof}[Proof of Corollary~\ref{main:c}]


In any case, there are at most finitely many closed orbits and any circuits with wandering holonomy contain at most finitely many non-wandering orbits because of Lemma~\ref{lem:circuit_w_wh_f}. 
Obviously, the assertion {\rm(2)} implies the assertion {\rm(3)}. 
By the compactness of $S$, since a circuit with wandering holonomy has a small neighborhood which is a M\"obius band, there are at most finitely many circuits with wandering holonomy. 

Suppose that $\Omega(v)$ consists of finitely many orbits.
By \cite[Lemma 3.1]{yokoyama2019properness}, since $\mathrm{R}(v) \subseteq \Omega(v)$, the finiteness of $\Omega(v)$ implies that $\mathrm{R}(v) = \emptyset$ and that $\Cv$ consists of finitely many orbits. 
Lemma~\ref{lem:wandering_pt} implies that any non-recurrent point in $\Omega(v)$ is contained in either a strict limit circuit or a circuit with wandering holonomy. 
This means that $\Omega(v)$ is the finite union of closed orbits, a strict limit circuit, and a circuit with wandering holonomy, and that any non-periodic limit circuits are strict limit circuits. 
Therefore the assertion {\rm(2)} holds. 

Suppose that $\mathrm{R}(v) = \emptyset$, any limit circuits consist of at most finitely many orbits, and there are at most finitely many closed orbits and strict limit circuits.
Lemma~\ref{lem:wandering_pt} implies that any non-recurrent point in $\Omega(v)$ is contained in either a limit circuit or a circuit with wandering holonomy. 
This implies that $\Omega(v)$ consists closed orbits, a strict limit circuit, and a circuit with wandering holonomy, and that any non-periodic limit circuits are strict limit circuits. 
Since any limit circuits consist of at most finitely many orbits, the finite existence of closed orbits, strict limit circuits, and circuits with wandering holonomy implies that $\Omega(v)$ consists of finitely many orbits. 
\end{proof}

\subsection{Correspondence between non-wandering properties and closedness}
We have the following statements. 

\begin{lemma}\label{lem:correspond_pn}
Let $v$ be a flow with finitely many connected components of the singular point set  on a compact surface $S$.
Suppose that $\mathop{\mathrm{Cl}}(v)$ is closed and there are neither non-closed recurrent points,  limit quasi-circuits, nor blow-ups of circuits with wandering holonomy. 
Then $\mathop{\mathrm{Cl}}(v) = \Omega(v)$. 
\end{lemma}

\begin{proof}
Since $\mathop{\mathrm{Cl}}(v)$ is closed, the blow-downs of any strict limit quasi-circuits are  
Corollary~\ref{thm:main03} implies $\overline{\mathop{\mathrm{Cl}}(v)} = \Omega(v)$. 
Since $\mathop{\mathrm{Cl}}(v)$ is closed, we obtain $\mathop{\mathrm{Cl}}(v) =\overline{\mathop{\mathrm{Cl}}(v)} =  \Omega(v)$. 
\end{proof}

Since $S - \Cv = \mathrm{P}(v)$ for any flow $v$ without non-closed recurrent orbits on a surface, Proposition~\ref{prop:necessary_dense02} and Lemma~\ref{lem:correspond_pn} imply Corollary~\ref{main_cor:a}. 
Corollary~\ref{main_cor:a}, Lemma~\ref{lem:blow_down02}, and Corollary~\ref{thm:main} imply the following correspondence. 

\begin{corollary}\label{thm:main02}
The following conditions are equivalent for a flow $v$ with finitely many connected components of the  singular point set on a surface $S$ with finite genus and finite ends:
\\
$(1)$ $\mathop{\mathrm{Cl}}(v) = \Omega(v)$.
\\
$(2)$ The closed point set $\mathop{\mathrm{Cl}}(v)$ is closed, and there are neither non-closed recurrent points, non-periodic virtually limit quasi-circuits, nor virtual blow-up of circuits with wandering holonomy.
\\
$(3)$ The closed point set $\mathop{\mathrm{Cl}}(v)$ is closed, every orbit is proper, and there are neither non-periodic virtually limit quasi-circuits nor virtual blow-up of circuits with wandering holonomy.
\\
$(4)$ The closed point set $\mathop{\mathrm{Cl}}(v)$ is closed, and the orbit space $S/v$ is $T_0$ and there are neither non-periodic virtually limit quasi-circuits nor virtual blow-up of circuits with wandering holonomy.

In any case, the set $\mathrm{P}(v)$ of non-recurrent points are open. 
\end{corollary}

\section{Examples} 

The non-compactness is necessary for the equivalence in the main results. 

\begin{example}\label{ex:counter_exam01}
There is a flow $v$ on a non-compact surface $S$ with $\emptyset \neq \mathrm{R}(v) \subset \overline{\mathop{\mathrm{Cl}}(v)} = \Omega(v)$. 
To construct such an example, recall the Denjoy construction. 
Fix an irrational number $r \in \R - \Q$, a translation $\widetilde{f_0} \colon \R \to \R$ by $\widetilde{f_0} (x) := x +r$, and an irrational rotation $f_0 \colon \mathbb{S}^1 \to \mathbb{S}^1$ on a circle $\mathbb{S}^1 := \R/\Z$ by $f_0([x]) := [x+r]$. 
Fix a sequence $(n_m)_{m \in \Z}$ with $n_0 = 0$ and $n_m < n_{m+1}$ such that $\lim_{m \to -\infty} f_0^{n_m}([0]) = [0]$ from negative side and $\lim_{m \to \infty} f_0^{n_m}([0]) = [0]$ from positive side. 
Consider a sequence $\{l_n := 1/(1+n^2) \mid n \in \Z \}$ and a sequence $\{I_n \mid n \in \Z \}$ of half-open intervals $I_n := [-1/(2+2n^2), 1/(2+2n^2))$ with the length $l_n$.  
Put $L := \sum_{n \in \Z} l_n = \sum_{n \in \Z} 1/(1+n^2) < \infty$ and define $y_n := \widetilde{f_0}^n(0) - \lfloor \widetilde{f_0}^n(0) \rfloor \in [0,1)$. 
Insert the interval $I_n$ at the point $y_n \in [0,1)$ and denote by $S^1 := \R/(1+L)\Z$ the resulting circle and by $f \colon S^1 \to S^1$ the resulting homeomorphism with a unique minimal set $\mathcal{C}$, which is a Cantor set. 
Let $v_0$ be the suspension flow of $f \colon S^1 \to S^1$ on the mapping torus $T^2 := (S^1 \times [0,1])/\sim$ with the unique minimal set $\mathcal{M}$, which is the suspension of the Cantor set. 

Define points $0_{n,-}$ and $0_{n,+}$ by the points of the boundary of $I_n$. 
By construction, the sequences $(0_{n_m,-})_{n \in \Z}$ and $(0_{n_m,+})_{n \in \Z}$ of points in $S^1$ satisfies that $\lim_{ m \to -\infty} 0_{n_m,-} = \lim_{n \to -\infty} 0_{n_m,+} = 0_{0,-}$ and $\lim_{m \to \infty} 0_{n_m,-} = \lim_{m \to \infty} 0_{n_m,+}  = 0_{0,+}$. 
Fix a monotone decreasing sequence $(r_l)_{l \in \Z_{>0}}$ of points in $I_0$ converging to $0_{0,-}$. 
Put $p_l := f^{n_l}(r_{|l|}) \in I_{n_l} \subset S^1$ for any $l \in \Z_{>0}$.
Then $\lim_{l \to -\infty} p_l = 0_{0,-}$ and $\lim_{l \to \infty} p_l  = 0_{0,+}$. 
Take pairwise disjoint closed intervals $B_l \subset I_0$ which are \nbds of points $r_l$. 
Then the images $D_l :=  f^{n_l}(B_{|l|}) \subset I_{n_l}$ for any $l \in \Z_{\neq 0}$ are \nbds of $p_l$. 
%
Removing a subset $( \{ 0_{0,-}, 0_{0,+} \} \sqcup \bigsqcup_{l \in \Z_{\neq 0}} D_ {l}) \times \{ 1/2 \} \subset T$, 
the resulting flow $v_1$ on the resulting surface $T_1$ satisfies that every minimal set is either $\mathcal{M}_1 := \mathcal{M} - (\{ 0_{0,-}, 0_{0,+} \} \times \{ 1/2 \})$
or a proper orbit $O_{v_1}(p)$ for some $p \in \bigsqcup_{l  \in \Z_{>0}} B_l \times \{ 1/2 \} \subset I_0 \times \{ 1/2 \}$. 
Consider open intervals $C_l \subset D_l$ containing $p_l$ with $C_{|l|} = f_1^{n_ {|l|} - n_{-|l|}}(C_{-|l|})$ for any $l \in \Z_{\neq 0}$. 
Identifying $C_l$ and $C_{-l}$, the resulting surface is an open surface with infinitely many genus and the resulting flow $v$ is non-singular such that every minimal set is either $\mathcal{M}_1$ or a proper orbit $O_{v_1}(p)$ for some $p \in \bigsqcup_{l  \in \Z_{>0}} B_l \times \{ 1/2 \} \subset I_0 \times \{ 1/2 \}$. 
Then $\mathcal{M}_1 = \mathrm{E}(v)$ and $\mathop{\mathrm{Per}}(v) = v(\bigcup_{l \in \Z_{>0}} C_l)$. 
Since $C_l \subset D_l =  f^{n_l}(B_{|l|})$, $\lim_{l \to -\infty} n_l = - \infty$, and $\lim_{l \to \infty} n_l =  \infty$, we obtain $\mathcal{M} \subset \overline{\mathop{\mathrm{Per}}(v)}$. 
The non-existence of locally dense orbits implies that $\emptyset \neq \mathrm{R}(v) \subset \overline{\mathop{\mathrm{Per}}(v)}$. 
\end{example}

We describe concrete examples and (necessary and) sufficient conditions for the denseness of closed orbits with respect to some kind of classes of flows. 

\begin{example}\label{ex:counter_exam_strict}
There is a flow $v$ generated by an $\Omega$-stable vector field with four critical points on a sphere with $\overline{\mathop{\mathrm{Cl}}(v)} \subsetneq \Omega(v)$ as in Figure~\ref{fig:omega}.
\begin{figure}
\begin{center}
\includegraphics[scale=0.3]{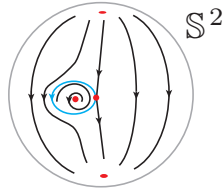}
\end{center}
\caption{An $\Omega$-stable spherical flow $v$ with $\overline{\mathop{\mathrm{Cl}}(v)} \subsetneq \Omega(v)$ such that the singular point set consists of one sink, two sources, and one saddle.}
\label{fig:omega}
\end{figure} 
In fact, the singular point set $\Sv$ consists of one sink, two sources, and one saddle as in Figure~\ref{fig:omega}. 
Then the unique homoclinic orbit of the saddle is the difference $\Omega(v) - \overline{\mathop{\mathrm{Cl}}(v)}$. 
Roughly speaking, it is known that $\Omega$-stable flows on closed surfaces are ``Morse-Smale'' flow without the non-existence condition of heteroclinic separatrices. 
Recall that a $C^1$ vector field $X$ on a manifold $M$ is $\Omega$-stable if there is its $C^1$ \nbd $\mathcal{U} \subset \chi^1(M)$ such that for any vector filed $Y$ in $\mathcal{U}$ there is a homeomorphism $h : \Omega(X) \to \Omega(Y)$ which maps orbits of $X$ to orbit of $Y$ preserving the orientation of orbits. 
A $C^1$ vector field on a closed manifold is $\Omega$-stable if and only if it satisfies the no-cycle condition  and all non-wandering orbits are hyperbolic closed orbits \cite{hayashi1997connecting,pugh1970omega}. 
Therefore any $\Omega$-stable flows on closed surfaces correspond to ``Morse-Smale'' flows without the non-existence condition of heteroclinic separatrices, and so the flow $v$ is $\Omega$-stable. 
\end{example}

The following two statements imply the necessity of the finite existence of connected components of the singular point set. 

\begin{lemma}\label{lem:counter_exam00}
There is a toral flow $v$ with uncountable singular points satisfying the following three conditions: 
\\
$(1)$ $\mathrm{R}(v) = \emptyset$ and there are no circuits. 
\\
$(2)$ $\overline{\mathop{\mathrm{Cl}}(v)} \neq \Omega(v)$.
\end{lemma}

\begin{proof}
%
Consider the suspension flow $v$ of a Denjoy homeomorphism $f \colon \mathbb{S}^1 \to \mathbb{S}^1$ with a minimal set $\mathcal{C} \subset \mathbb{S}^1$, which is a Cantor set, on the mapping torus $S := (\mathbb{S}^1 \times [0,1])/ (x,0) \sim (f(x),1)$. 
Then the minimal set $\mathcal{M}$ satisfies that $\mathcal{M} \cap (\mathbb{S}^1 \times {1/2}) = \mathcal{C} \times \{ 1/2 \} \subset S$. 
Moreover, we obtain $\Omega(v) = \mathrm{R}(v) = \mathcal{M}$. 
Since the complement $\mathbb{S}^1 - \mathcal{C}$ of the Cantor set $\mathcal{C}$ in the circle $\mathbb{S}^1$ consists of wandering domaines and so consists of non-recurrent points, we have $S - \mathcal{M} = \mathrm{P}(v)$. 
Using the bump function $\varphi$ with $\varphi^{-1}(0) = \mathcal{C} \times \{ 1/2 \}$, replace the orbits in $\mathcal{M}$ with a union of singular points and multi-saddle separatrices of the resulting flow $v_{\varphi}$ (i.e. $\mathcal{M} = \mathop{\mathrm{Sing}(v_{\varphi})} \sqcup \{ \text{separatrix of } v_{\varphi} \}$) such that $O_{v_{\varphi}}(y) = O_{v}(y)$ for any point $y \in S - \mathcal{M}$.
Then $S =  \mathop{\mathrm{Sing}}(v_{\varphi}) \sqcup \mathrm{P}(v_{\varphi})$ and so $\overline{\mathop{\mathrm{Cl}}(v)} = \mathop{\mathrm{Sing}}(v)  = \mathcal{C} \times \{ 1/2 \} \subsetneq \mathcal{M} = \Omega(v_{\varphi})$. 
\end{proof}


\begin{lemma}\label{lem:counter_exam01}
There is a flow $v$ on a closed surface $S$ with countable singular points satisfying the following three conditions: 
\\
$(1)$ $\mathrm{R}(v) = \emptyset$ and there are neither strict limit circuits nor circuits with wandering holonomy.
\\
$(2)$ $\overline{\mathop{\mathrm{Cl}}(v)} = \Cv = \Sv \neq \Omega(v)$.
\\
$(3)$ The flow $v$ has a non-periodic non-limit circuit in $\Omega(v)$. 
\end{lemma}

\begin{proof}
Consider a toral flow $w$ which consists of one non-contractible limit cycle $C$ and non-closed proper orbits. 
Then $\T^2 = C \sqcup \mathrm{P}(w)$. 
Fix a point $z \in C$ and a non-closed proper orbit $O$.
Write an open trivial flow box $D := \T^2 -(C \sqcup O) \subset \mathrm{P}(w)$.
Choose a closed transversal $T$ through $z$,  a point $x \in O$, and a monotonic sequence $(t_n)_{n \in \Z}$ on $O$ with $T \cap O = \{ x_n \}_{n \in \Z}$, $\lim_{n \to \infty} t_n = \infty$, $\lim_{n \to - \infty} t_n = - \infty$, and $\lim_{n \to \infty} x_n = \lim_{n \to - \infty} x_n = z$ such that any connected component of the set difference $D \setminus T$ are open trivial flow boxes, where $x_n := w_{t_n}(x)$.
Write $D_n$ the open trivial flow box with four corners $x_n, x_{n+1}, x_{n+2}, x_{n+3}$ such that $D \setminus T = \bigcup_{n \in \Z} D_n$.
Write an open trivial flow box $D'_{2n} = D_{2n} \sqcup D_{2n + 1} \sqcup ((T \setminus O) \cap \overline{D_{2n}} \cap \overline{D_{2n+1}}) \subset D$.
Replacing the closure of each $D'_{2n}$ by a box with a flow as shown in Figure~\ref{Spiral}, we obtain the resulting flow $v$ on the torus $\T^2$ such that the singular point set $\mathop{\mathrm{Sing}}(v) = \{ z \} \sqcup \{ x_n, y_n \}_{n \in \Z} \subset T$ is countable. 
Since any open trivial flow box $D'_{2n}$ consists of singular points and non-recurrent points, we have that $\T^2 = \mathop{\mathrm{Sing}}(v) \sqcup \mathrm{P}(v)$ and $\Omega(v) = C \sqcup \{ x_n, y_n \}_{n \in \Z} \supsetneq  \mathop{\mathrm{Sing}}(v) = \overline{\mathop{\mathrm{Cl}}}(v)$.
The non-periodic circuit $C$ is neither the $\omega$-limit set of a point nor the $\alpha$-limit set of a point and so is not a limit circuit. 
Moreover, since the first return maps of any closed transverse arc are orientation-preserving, there are no circuits with wandering holonomy. 
Since any strict limit circuits are contained in $\Omega(v)$, there are no limit circuits. 
\end{proof}
\begin{figure}
\begin{center}
\includegraphics[scale=0.4]{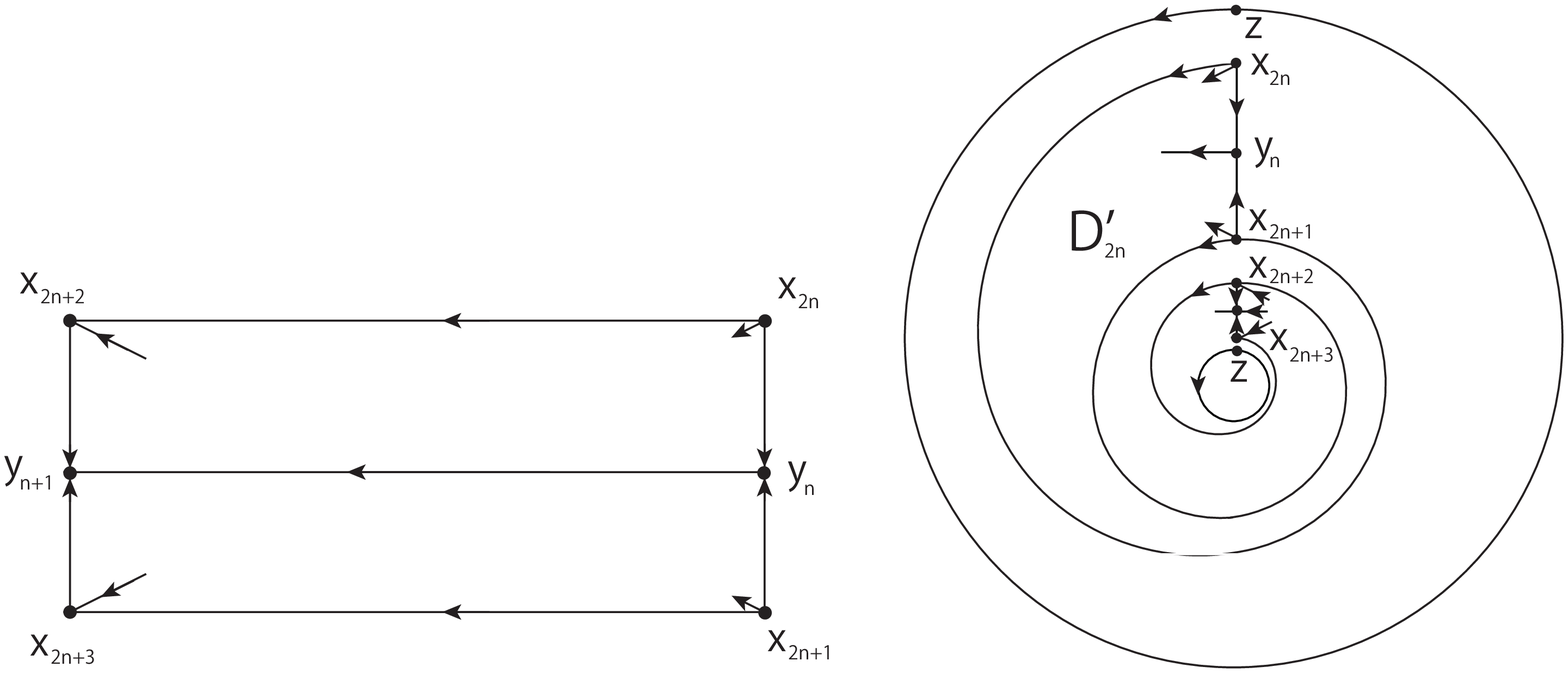}
\end{center}
\caption{An open flow box $D'_{2n}$}
\label{Spiral}
\end{figure}
%
The following example implies the necessity of the non-existence of circuits with wandering holonomy in the results.

\begin{lemma}\label{lem:counter_exam02}
There is a flow $w$ without non-degenerate singular points on a non-orientable closed surface $S$ with non-orientable genus four satisfying the following three conditions: 
\\
{\rm(1)} $\overline{\mathop{\mathrm{Cl}}(w)} = \mathop{\mathrm{Cl}}(w) = \mathop{\mathrm{Sing}}(w) = \bigcup_{x \in S} \omega(x) \cup \alpha(x)\subsetneq \Omega(w)$. 
\\
{\rm(2)} $\mathop{\mathrm{Per}}(w) \sqcup \mathrm{R}(v) = \emptyset$. 
\\
{\rm(3)} There are no strict limit quasi-circuits but a circuit with wandering holonomy. 
\end{lemma}

\begin{proof}
Consider a flow $v_0$ on a non-orientable compact surface $S_0$ as in Figure~\ref{g2-ex} with $S_0 =  \mathop{\mathrm{Sing}}(w_0) \sqcup \mathrm{P}(v)$ and $\bigcup_{x \in S_0} \omega(x) \cup \alpha(x) = \mathop{\mathrm{Sing}}(w_0) \subsetneq \mathop{\mathrm{Sing}}(w_0) \sqcup O' = \Omega(w_0)$ such that $\Omega(w_0)$ does not contain limit circuits, where $O'$ is a proper orbit. 
Then the lift of the flow $w_0$ to the double $S$ of $S_0$ is desired. 
\begin{figure}
\begin{center}
\includegraphics[scale=0.3]{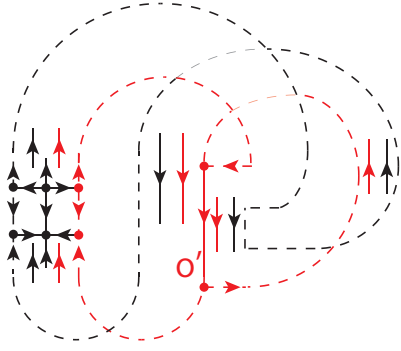}
\end{center}
\caption{A flow on a compact surface with two non-orientable genus and two boundary components consists of one sink, one source, six $\partial$-saddles, and non-closed proper orbits.}
\label{g2-ex}
\end{figure}
\end{proof}

Recall a circuit $\gamma$ has non-orientable holonomy if there are a non-singular point $x \in \gamma$ and arbitrarily small open transverse arc $I$ containing $x$ such that the domain of the first return map on $I$ is not empty but non-orientable. 
The non-existence of circuits with wandering holonomy in the results can not be weakened. 
In fact, the density condition does not inhibit the existence of non-orientable holonomy as follows. 

\begin{lemma}\label{lem:counter_exam03}
There is a flow $v$ on a Klein bottle satisfying $\overline{\mathop{\mathrm{Cl}}(v)} = \Omega(v)$ with a circuit with fixed-point-free non-orientable holonomy. 
\end{lemma}

\begin{proof}
Consider a flow on a Klein bottle consisting of periodic orbits. 
Replacing a periodic orbit with non-orientable holonomy with a $0$-saddle with a homoclinic separatrix, we have the resulting flow $v$ which has a circuit with fixed-point-free non-orientable holonomy and consists of one non-closed proper orbits and closed orbits such that $\mathbb{K} = \Omega(v) = \overline{\mathop{\mathrm{Cl}}(v)}$. 
\end{proof}

\section{Observations on the density for non-wandering flows on compact surfaces}
Finally, we observe that the denseness of closed orbits for non-wandering flows on compact surfaces corresponds to the non-existence of locally dense orbits, and that any Hamiltonian flows on compact surfaces and any gradient flows on manifolds satisfy the denseness of closed orbits. 
We have the following characterization of non-wandering flows to characterize the density condition. 

\begin{proposition}\label{interior}
The following are equivalent for a flow $v$ on a compact surface $S$: 
\\
{\rm(1)} The flow $v$ is  non-wandering. 
\\
{\rm(2)} $\mathrm{int}\mathrm{P}(v) = \emptyset$. 
\\
{\rm(3)}  $\mathrm{int}\mathrm{P}(v) \sqcup \mathrm{E}(v) = \emptyset$ $(\mathrm{i.e.}$ $S = \mathrm{Cl}(v) \sqcup (\mathrm{P}(v) - \mathrm{int}\mathrm{P}(v)) \sqcup \mathrm{LD}(v) )$. 
\\
{\rm(4)} $S = \overline{\mathop{\mathrm{Cl}}(v) \sqcup \mathrm{LD}(v)}$. 
\end{proposition}

\begin{proof} 
The assertion (3) implies the assertion (2). 
Recall that $S = \mathop{\mathrm{Cl}}(v) \sqcup \mathrm{P}(v) \sqcup \mathrm{R}(v)$ and that the union $\mathrm{P}(v)$ is the set of non-recurrent points.  
Then the assertions (3) and (4) are equivalent. 
By Lemma~\ref{lem:top23}, the union $\mathrm{P}(v) \sqcup \mathrm{E}(v)$ is a \nbd of $\mathrm{E}(v)$. 
By the Ma\v \i er theorem \cite{markley1970number,Torhorst1921}, the closure $\overline{\mathrm{E}(v)}$ is a finite union of closures of exceptional orbits and so is nowhere dense.  
This means that $\mathrm{E}(v) = \emptyset$ if $\mathrm{int}\mathrm{P}(v) = \emptyset$. 
Therefore the assertion $(2)$ implies the assertion $(3)$. 

Suppose that $v$ is non-wandering. 
By \cite[Theorem III.2.12 and Theorem III.2.15]{BS}, the set of recurrent points is dense in $S$. 
The density of recurrent points implies that $\mathrm{int}\mathrm{P}(v) = \emptyset$ and so $\mathrm{E}(v) = \emptyset$.  
This implies that $S = \overline{\mathop{\mathrm{Cl}}(v) \sqcup \mathrm{LD}(v)}$. 

Conversely, suppose that $\mathrm{int}\mathrm{P}(v) \sqcup \mathrm{E}(v)  = \emptyset$. 
Then $S = \mathop{\mathrm{Cl}}(v) \sqcup (\mathrm{P}(v) - \mathrm{int}\mathrm{P}(v)) \sqcup \mathrm{LD}(v)$. 
Therefore the closure of the set of recurrent points is the whole surface (i.e. $S = \overline{\mathop{\mathrm{Cl}}(v) \sqcup \mathrm{LD}(v)}$) and so $v$ is  non-wandering. 
\end{proof}

By the previous proposition, notice that exceptional quasi-minimal sets imply the existence of wandering domains. 
We have the following equivalence. 

\begin{lemma}\label{lem-01}
The following are equivalent for a flow $v$ on a compact surface $S$:
\\
$(1)$ The flow $v$ is non-wandering and $\overline{\mathop{\mathrm{Cl}}(v)} = \Omega(v)$.
\\
$(2)$ $\overline{\mathop{\mathrm{Cl}}(v)} = S$.
\\
$(3)$ $\mathrm{int} \mathrm{P}(v)  = \emptyset$ and $\overline{\mathop{\mathrm{Cl}}(v)} = \Omega(v)$.
\\
$(4)$ $\mathrm{LD}(v) \sqcup \mathrm{int} \mathrm{P}(v)  = \emptyset$ $(\mathrm{i.e.}$ $S = \mathop{\mathrm{Cl}}(v) \sqcup (\mathrm{P}(v) - \mathrm{int}\mathrm{P}(v))$ $)$.
\\
$(5)$ The flow $v$ is non-wandering and $\mathrm{LD}(v) = \emptyset$.
\end{lemma}

\begin{proof}
Recall that $S = \mathop{\mathrm{Cl}}(v) \sqcup \mathrm{P}(v) \sqcup \mathrm{R}(v)$.
Obviously, the assertions $(1)$ and $(2)$ are equivalent, and the assertion $(4)$ implies the assertion $(2)$.
Lemma~\ref{thm41} implies that the assertion $(3)$ implies the assertion $(4)$.
We show that the assertion $(2)$ implies the assertion $(3)$.
Indeed, if $\overline{\mathop{\mathrm{Cl}}(v)} = S$, then $\mathrm{int} \mathrm{P}(v)  = \emptyset$ and $\Omega(v) \subseteq S =  \overline{\mathop{\mathrm{Cl}}(v)} \subseteq \Omega(v)$ because of definition of $\Omega(v)$. 
We show that the assertion $(4)$ is equivalent to the assertion $(2)$.
Indeed, suppose $\mathrm{LD}(v) \sqcup \mathrm{int} \mathrm{P}(v)  = \emptyset$.
By Lemma~\ref{lem:top23}, we have $\mathrm{E}(v) = \emptyset$ and so $S = \mathop{\mathrm{Cl}}(v) \sqcup (\mathrm{P}(v) - \mathrm{int}\mathrm{P}(v)) = \overline{\mathop{\mathrm{Cl}}(v)}$.

Notice that the assertions $(1)$--$(4)$ imply the assertion $(5)$. 
Proposition~\ref{interior} implies the assertion (5) implies the assertion $(4)$. 
\end{proof}

In the non-wandering case, the denseness of closed orbits is characterized as follows. 

\begin{corollary}\label{prop:nw}
The following are equivalent for a non-wandering flow $v$ on a compact surface:
\\
$(1)$ $\overline{\mathop{\mathrm{Cl}}(v)} = \Omega(v)$.
\\
$(2)$ $\mathrm{LD}(v) = \emptyset$.
\end{corollary}

The non-wandering property in the previous corollary is necessary. 
In fact, the Denjoy flow has no locally dense orbits but the non-wandering set consists of exceptional recurrent orbits. 
This implies the following observation. 

\begin{corollary}
Each Hamiltonian flow on a compact surface satisfies the denseness of closed orbits. 
\end{corollary}

\begin{proof}
Let $v$ be a Hamiltonian flow on a compact surface. 
Since any Hamiltonian flow on a compact surface is an area-preserving and so has no wandering domain, the flow $v$ is non-wandering. 
The existence of the Hamiltonian of $v$ on the surface implies the non-existence of non-closed recurrent  orbits. 
Corollary~\ref{prop:nw} implies the assertion. 
\end{proof}

The compactness in the previous corollary is necessary. 
In fact, the flow generated by a vector field $X = (1,0)$ on the plane $\R^2$ is Hamiltonian but consists of non-recurrent orbits.

\bibliographystyle{my-amsplain-nodash-abrv-lastnamefirst-nodot}
\bibliography{yt20210901}

\end{document}